\documentclass[A4wide,10pt]{amsart}

\usepackage{kotex}
\usepackage{float}
\usepackage{amsmath,amsthm,amscd,amsfonts,amssymb,epic,eepic,bbm,graphicx,ytableau}
\usepackage{rotating}
\usepackage{subcaption}

\usepackage{tikz}
\usepackage{tkz-euclide}
\usetikzlibrary {positioning}
\usetikzlibrary{backgrounds}
\usetikzlibrary{angles,quotes}
\usepackage{esvect}

\newtheorem{thm}{Theorem}[section]
\newtheorem{cor}[thm]{Corollary}

\newtheorem{lem}[thm]{Lemma}

\newtheorem{prop}[thm]{Proposition}

\newtheorem{rem}[thm]{Remark}
\newtheorem{ex}[thm]{Example}

\newcommand{\la}{\lambda}
\newcommand{\lf}{\lfloor}
\newcommand{\rf}{\rfloor}
\newcommand{\F}{\mathcal{F}}
\newcommand{\M}{\mathcal{M}}
\newcommand{\Fc}{\mathcal{F}_{\rm c}}
\newcommand{\MC}{\mathcal{M}_{\rm c}}
\newcommand{\oF}{\overline{\mathcal{F}}}
\newcommand{\oM}{\overline{\mathcal{M}}}
\newcommand{\oFc}{\overline{\mathcal{F}}_{\rm c}}
\newcommand{\oMc}{\overline{\mathcal{M}}_{\rm c}}
\newcommand{\run}{\operatorname{run}}
\newcommand{\MD}{\operatorname{MD}}
\newcommand{\cc}{\operatorname{cc}}
\newcommand{\scc}{\operatorname{scc}}
\title[Combinatorics on bounded free Motzkin paths and its applications]{Combinatorics on bounded free Motzkin paths and its applications}

\author{Hyunsoo Cho}
\address{Hyunsoo Cho, Institute of Mathematical Sciences, Ewha Womans University, Seoul, Republic of Korea}
\email{hyunsoo@ewha.ac.kr}

\author{JiSun Huh}
\address{JiSun Huh, Department of Mathematics, Ajou University, Suwon, Republic of Korea}
\email{hyunyjia@ajou.ac.kr}

\author{Hayan Nam}
\address{Hayan Nam, Department of Mathematics, Duksung Women's University,  Seoul, Republic of Korea}
\email{hnam@duksung.ac.kr}

\author{Jaebum Sohn}
\address{Jaebum Sohn, Department of Mathematics, Yonsei University, Seoul, Republic of Korea}
\email{jsohn@yonsei.ac.kr}

\begin{document}

\maketitle

\begin{abstract}
In this paper, we construct a bijection from a set of bounded free Motzkin paths to a set of bounded Motzkin prefixes that induces a bijection from a set of bounded free Dyck paths to a set of bounded Dyck prefixes. We also give bijections between a set of bounded cornerless Motzkin paths and a set of $t$-core partitions, and a set of bounded cornerless symmetric Motzkin paths and a set of self-conjugate $t$-core partitions. As an application, we get explicit formulas for the number of ordinary and self-conjugate $t$-core partitions with a fixed number of corners.
\end{abstract}


\section{Introduction}\label{sec:intro}

The main result of this paper is finding a bijection between two sets of paths in a bounded strip, which have been studied by several researchers (for example, see \cite{AK,Cigler,CK,Dershowitz,GP,KM}).

A \emph{Motzkin path} of length $n$ is a path from $(0,0)$ to $(n,0)$ which stays weakly above the $x$-axis and consists of steps $u=(1,1)$, $d=(1,-1)$, and $f=(1,0)$, called \emph{up}, \emph{down}, and \emph{flat} steps, respectively.
A \emph{free Motzkin path} of length $n$ is a path which starts at \((0,0)\) or \((0,1)\), ends at \((n,0)\), and consists of \(u\), \(d\), and \(f\). A Motzkin path with no restrictions on the end point is called a Motzkin prefix. For a given path, a \emph{peak} is a point preceded by an up step and followed by a down step and a \emph{valley} is a point preceded by a down step and followed by an up step. We say that a path is \emph{cornerless} if it has no peaks or valleys.

For non-negative integers \(m,r\), and \(k\), let \(\F(m,r,k)\) be the set of free Motzkin paths of length \(m+r\) with \(r\) flat steps that are contained in the strip \(-\lf \frac{k}{2}\rf \leq y \leq \lf \frac{k+1}{2} \rf\).
We denote \(\M(m,r,k)\) the set of Motzkin prefixes of length \(m+r\) with \(r\) flat steps that are contained in the strip \(0\leq y \leq k\). 
We define \(L_k\) to be one of the boundaries of each path depending on the value of $k$. More specifically, for $P\in \F(m,r,k)$,
denote \(L_k\) by  
\[
y=\begin{cases}
\lf\frac{k+1}{2}\rf & \text{if \(k\) is odd},\\ 
-\lf\frac{k}{2}\rf & \text{if \(k\) is even}.
\end{cases}
\]

Let \(\oF(m,r,k)\) (resp.  \(\oM(m,r,k)\)) be the set of paths in \(\F(m,r,k)\) (resp. \(\M(m,r,k)\)) which touch the line \(L_k\) (resp. $y=k$) so that
\[
\F(m,r,k)=\bigcup_{i=0}^{k}\oF(m,r,i) \quad \text{and} \quad \M(m,r,k)=\bigcup_{i=0}^{k}\oM(m,r,i).
\]

Our main theorem states the following. 

\begin{thm} \label{thm:main}
For given non-negative integers \(m,r\), and \(k\), there is a bijection between the sets \(\oF(m,r,k)\) and \(\oM(m,r,k)\).
\end{thm}

To prove Theorem~\ref{thm:main}, we construct a map \(\phi_{m,k}\) and show that it is bijective in Sections \ref{sec:map_phi} and \ref{sec:1-1}.

Using the adjacency matrices of path graphs, Cigler \cite{Cigler} showed that
\[
|A_{n,k}|=|B_{n,k}|=\sum_{j\in\mathbb{Z}}(-1)^j \binom{n}{\lf \frac{n+(k+2)j}{2}\rf}
\]
and expected the existence of a simple bijection between \(A_{n,k}\) and \(B_{n,k}\), 
where $A_{n,k}$ is the set of paths of length $n$ which consist of $u$ and $d$ only, start at $(0,0)$, end on height $0$ or $-1$, and are contained in the strip $-\lf \frac{k+1}{2}\rf\le y\le \lf\frac{k}{2}\rf$ of width $k$, and $B_{n,k}$ is the set of paths of length $n$ which consist of $u$ and $d$ only, start at $(0,0)$ and are contained in the strip $0\le y\le k$. Recently, Gu and Prodinger \cite{GP} and Dershowitz \cite{Dershowitz} found bijections between $A_{n,k}$ and $B_{n,k}$ independently. We note that Theorem~\ref{thm:main} with no flat step (equivalently, $r=0$) gives a new bijection between $A_{n,k}$ and $B_{n,k}$ since
$\F(n,0,k)$ can be obtained from $A_{n,k}$ by mirroring left and right and flipping along the $x$-axis, and $\M(n,0,k)=B_{n,k}$ as it is. 
We should mention that the bijection \(\phi_{m,k}\) is inspired by the bijection due to Gu and Prodinger, but there is a property that \(\phi_{m,k}\) holds whereas Gu and Prodinger's does not. This property is described in Section \ref{sec:special}.

Let \(\oFc(m,r,k)\) be the set of cornerless free Motzkin paths in \(\oF(m,r,k)\) that never start with a down (resp. up) step for odd (resp. even) \(m\) and \(\oMc(m,r,k)\) be the set of cornerless Motzkin prefixes in \(\oM(m,r,k)\) that end with a flat step. In Section~\ref{sec:cornerless}, we show that \(\phi_{m,k}\) induces a bijection between \(\oFc(m,r,k)\) and \(\oMc(m,r,k)\).

In Section~\ref{sec:tcore}, we combinatorially interpret $t$-core partitions by cornerless Motzkin paths. We describe a bijection between a set of cornerless Motzkin paths and a set of $t$-core partitions. As an application of this bijection, we count the number of $t$-core partitions with $m$ corners. 
In Section~\ref{sec:self}, we also count the number of self-conjugate $t$-core partitions with $m$ corners by constructing bijections between any pair of the following sets: a set of cornerless free Motzkin paths, a set of cornerless symmetric Motzkin paths, and a set of self-conjugate $t$-core partitions.


\section{Bijection}\label{sec:bijection}

In this section, we recursively define a map 
\[
\phi_{m,k}: \bigcup_{r\geq 0}\oF(m,r,k) \rightarrow \bigcup_{r\geq 0}\oM(m,r,k),
\]
according to the values of \(m\) and \(k\), and then show that it is bijective. 
For simplicity, we define some notations first. For a path \(P=p_1p_2\dots p_n\), where each \(p_i\) denotes the \(i\)th step in \(P\), let 
\[
\overline{P}:=\overline{p}_1\overline{p}_2\dots\overline{p}_n \quad \text{and} \quad
\overleftarrow{P}:=\overline{p}_n\overline{p}_{n-1}\dots\overline{p}_1,
\]
where \(\overline{u}:=d,~ \overline{d}:=u\), and \(\overline{f}:=f\).

\subsection{Map \(\phi_{m,k}\)}\label{sec:map_phi}

Now we define the map. Let \(P\) be a path in the set \(\oF(m,r,k)\) for some $r\ge 0$, and \(\gamma\geq 0\) denote the maximum number such that \(f^{\gamma}\) is a suffix of \(P\).   
\begin{enumerate}
    \item[\textbf{Case 0}.] If \(k=0\) or \(k=1\), then the map is defined as
        \[
        \phi_{m,k}(P):=\overleftarrow{P}.
        \]
        We show the bijection \(\phi_{m,1}\) in Figure~\ref{fig:case0}. 
\end{enumerate}

\begin{figure}[hbt]
\small{
\begin{subfigure}[b]{0.4\textwidth}
\centering
\begin{tikzpicture}[scale=0.8]

\filldraw[fill=gray!30] (1,0) -- (3,0) -- (3,0.5) -- (1,0.5) --cycle;

\draw[gray, ->] (0,0) -- (4.5,0);
\draw[gray, ->] (0,0) -- (0,1);

\foreach \i in {1,...,8}
\draw[dotted] (0.5*\i,0) -- (0.5*\i,0.5)
;	

\draw[dotted] (0,0.5) -- (4,0.5);
\draw[thick] 
(0,0.5) -- (1,0.5) -- (1.5,0) 
(2.5,0.5) -- (3,0) -- (4,0) ;

\node[above] at (0.5,0.45) {\(f^{\beta}\)};
\node[above] at (2,0) {\(B\)};
\node[above] at (3.5,-0.05) {\(f^{\gamma}\)};

\end{tikzpicture}\\

\begin{tikzpicture}[scale=0.8]

\filldraw[fill=gray!30] (1,0) -- (3,0) -- (3,0.5) -- (1,0.5) --cycle;

\draw[gray, ->] (0,0) -- (4.5,0);
\draw[gray, ->] (0,0) -- (0,1);

\foreach \i in {1,...,8}
\draw[dotted] (0.5*\i,0) -- (0.5*\i,0.5)
;	

\draw[dotted] (0,0.5) -- (4,0.5);
\draw[thick] 
(0,0) -- (1,0) -- (1.5,0.5)
(2.5,0) -- (3,0.5) -- (4,0.5);

\node[above] at (0.5,-0.05) {\(f^{\gamma}\)};
\node[above] at (2,-0.07) {\(\overleftarrow{B}\)};
\node[above] at (3.5,0.45) {\(f^{\beta}\)};

\node at (2,1.2) {\(\updownarrow\)};
\end{tikzpicture}
\subcaption{For odd \(m\)}
\end{subfigure}
\qquad 
\begin{subfigure}[b]{0.4\textwidth}
\centering
\begin{tikzpicture}[scale=0.8]

\filldraw[fill=gray!30] (2.5,0) -- (4.5,0) -- (4.5,0.5) -- (2.5,0.5) --cycle;

\draw[gray, ->] (0,0) -- (6,0);
\draw[gray, ->] (0,0) -- (0,1);

\foreach \i in {1,...,11}
\draw[dotted] (0.5*\i,0) -- (0.5*\i,0.5)
;	

\draw[dotted] (0,0.5) -- (5.5,0.5);
\draw[thick] 
(0,0) -- (1,0) -- (1.5,0.5) -- (2.5,0.5) -- (3,0)
(4,0.5) -- (4.5,0) -- (5.5,0) ;

\node[above] at (0.5,-0.05) {\(f^{\epsilon}\)};
\node[above] at (1.15,0.2) {\(u\)};
\node[above] at (2,0.45) {\(f^{\beta}\)};
\node[above] at (3.5,0) {\(B\)};
\node[above] at (5,-0.05) {\(f^{\gamma}\)};

\end{tikzpicture}\\

\begin{tikzpicture}[scale=0.8]

\filldraw[fill=gray!30] (1,0) -- (3,0) -- (3,0.5) -- (1,0.5) --cycle;

\draw[gray, ->] (0,0) -- (6,0);
\draw[gray, ->] (0,0) -- (0,1);

\foreach \i in {1,...,11}
\draw[dotted] (0.5*\i,0) -- (0.5*\i,0.5)
;	

\draw[dotted] (0,0.5) -- (5.5,0.5);
\draw[thick] 
(0,0) -- (1,0) -- (1.5,0.5)
(2.5,0) -- (3,0.5) -- (4,0.5) -- (4.5,0) -- (5.5,0) ;

\node[above] at (0.5,-0.05) {\(f^{\gamma}\)};
\node[above] at (2,-0.07) {\(\overleftarrow{B}\)};
\node[above] at (3.5,0.45) {\(f^{\beta}\)};
\node[above] at (4.35,0.15) {\(d\)};
\node[above] at (5,-0.05) {\(f^{\epsilon}\)};
\node at (3,1.2) {\(\updownarrow\)};
\end{tikzpicture}
\subcaption{For even \(m\)}
\end{subfigure}
}
\caption{The bijections \(\phi_{m,1}\) in Case 0}
\label{fig:case0}
\end{figure}
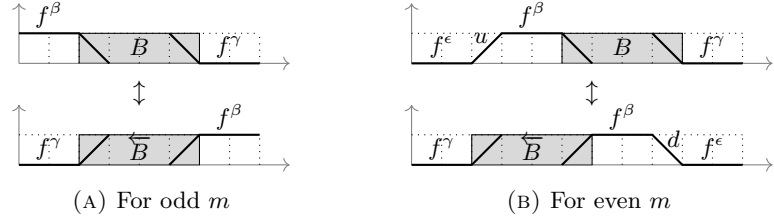

Now assume \(k>1\). 
A \emph{special step} of \(P\) is the first step ending on the line \(L_k\). 
We write \(P\) as 
\begin{equation}\label{eq:basic}
    P=Af^{\alpha}sf^{\beta}Bf^{\gamma},
\end{equation} 
where \(s\) is the special step, \(\alpha\geq 0\) (resp. \(\beta\geq 0\)) is the maximum number of consecutive flat steps right before (resp. after) the step \(s\), \(A\) denotes the prefix of \(P\) before the subpath \(f^{\alpha}s\), and \(B\) denotes the subpath between the subpaths \(sf^{\beta}\) and \(f^{\gamma}\). Note that \(A\) and \(B\) never end with a flat step (See Figure~\ref{fig:basic}).

\begin{figure}[hbt]
\small{
\begin{subfigure}[b]{0.45\textwidth}
\centering
\begin{tikzpicture}[scale=0.5]

\filldraw[fill=gray!30] (0,-1.5) -- (0,1.5) -- (2,1.5) -- (2,-1.5) --cycle;
\filldraw[fill=gray!30] (4.5,-1.5) -- (4.5,2) -- (6.5,2) -- (6.5,-1.5) --cycle;

\draw[gray, ->] (0,0) -- (8,0);
\draw[gray, ->] (0,-1.5) -- (0,2.5);
\draw[gray] (0,2) -- (7.5,2);

\foreach \i in {1,...,15}
\draw[dotted] (0.5*\i,-1.5) -- (0.5*\i,2)
;	

\foreach \i in {1,...,3}
\draw[dotted] (0,0.5*\i) -- (7.5,0.5*\i)
;	

\foreach \i in {1,...,3}
\draw[dotted] (0,-0.5*\i) -- (7.5,-0.5*\i)
;

\draw[thick] 
(1.5,1) -- (2,1.5) -- (3,1.5) -- (3.5,2) -- (4.5,2) -- (5,1.5)
(6,0.5) -- (6.5,0)
(6,-0.5) -- (6.5,0) -- (7.5,0);

\node at (1,0.75) {\(A\)};
\node[above] at (2.5,1.45) {\(f^{\alpha}\)};
\node[above] at (3.15,1.65) {\(s\)};
\node[above] at (4,1.95) {\(f^{\beta}\)};
\node at (5.5,1) {\(B\)};
\node[above] at (7,-0.05) {\(f^{\gamma}\)};
\node[left] at (0,2) {\(\frac{k+1}{2}\)};
\node[left] at (0,-1.5) {\(-\frac{k-1}{2}\)};
\node[right] at (7.5,2) {\(L_k\)};
\node[below] at (0,-2.5) {};

\end{tikzpicture}
\subcaption{For odd \(k\)}
\end{subfigure}
\qquad
\begin{subfigure}[b]{0.45\textwidth}
\centering
\begin{tikzpicture}[scale=0.5]

\filldraw[fill=gray!30] (0,-1.5) -- (0,2) -- (2,2) -- (2,-1.5) --cycle;
\filldraw[fill=gray!30] (4.5,-2) -- (4.5,2) -- (6.5,2) -- (6.5,-2) --cycle;

\draw[gray, ->] (0,0) -- (8,0);
\draw[gray, ->] (0,-2) -- (0,2.5);
\draw[gray] (0,-2) -- (7.5,-2);

\foreach \i in {1,...,15}
\draw[dotted] (0.5*\i,-2) -- (0.5*\i,2)
;	

\foreach \i in {1,...,4}
\draw[dotted] (0,0.5*\i) -- (7.5,0.5*\i)
;	

\foreach \i in {1,...,3}
\draw[dotted] (0,-0.5*\i) -- (7.5,-0.5*\i)
;

\draw[thick] 
(1.5,-1) -- (2,-1.5) -- (3,-1.5) -- (3.5,-2) -- (4.5,-2) -- (5,-1.5)
(6,0.5) -- (6.5,0)
(6,-0.5) -- (6.5,0) -- (7.5,0);

\node at (1,-0.75) {\(A\)};
\node[below] at (2.5,-1.45) {\(f^{\alpha}\)};
\node[below] at (3.15,-1.65) {\(s\)};
\node[below] at (4,-1.95) {\(f^{\beta}\)};
\node at (5.5,-1) {\(B\)};
\node[below] at (7,-0.05) {\(f^{\gamma}\)};
\node[left] at (0,2) {\(\frac{k}{2}\)};
\node[left] at (0,-2) {\(-\frac{k}{2}\)};
\node[right] at (7.5,-2) {\(L_k\)};

\end{tikzpicture}
\subcaption{For even \(k\)}
\end{subfigure}
}
\caption{The division of \(P\) for \(k>1\)}
\label{fig:basic}
\end{figure}
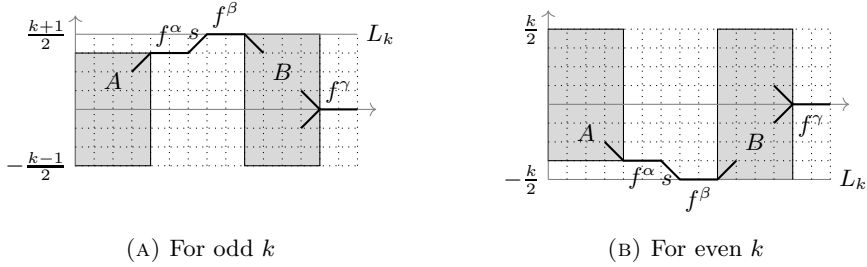

Let the last vertex on the line \(L_k\) (resp. $y=k$) be the \emph{turning point} of a path in \(\oF(m,r,k)\) (resp. \(\oM(m,r,k)\)). We call the first step after the turning point starting from the $x$-axis and heading away from the line $L_k$ the \emph{break step}, and denote it by $b$.
If \(P\) has the break step \(b\), let 
\(\delta\geq 0\) be the maximum number of consecutive flat steps right before the step \(b\) and we write \(B\) as \(B_1f^{\delta}bB_2\).

\begin{enumerate}
    \item[\textbf{Case 1}.] Let \(m\) and \(k\) have the same parity with \(k>1\). 
    \begin{enumerate}
        \item[{\rm i)}] If there is no break step, then we write \(P\) as \eqref{eq:basic} and define the map as
        \begin{equation}\label{eq:case1-1}
            \phi_{m,k}(P):=\begin{cases}
            Q & \text{if \(k\) is odd},\\
            \overline{Q} & \text{if \(k\) is even},
            \end{cases}
        \end{equation}
        where
        \[
        Q:=f^{\gamma}\overline{B}f^{\alpha}sAf^{\beta}.
        \]
        Note that \(\phi_{m,k}(P)\) ends on the line \(y=k\).
        \item[{\rm ii)}] If there is the break step \(b\), then \(P\) can be written as  
        \begin{equation}\label{eq:case1-2P}
           P=Af^{\alpha}sf^{\beta}B_1f^{\delta}bB_2f^{\gamma}.
        \end{equation}
        Note that \(B_1\) is a subpath starting from the line \(L_k\) and ending at the \(x\)-axis with a down (resp. up) step, and \(B_2\) is a subpath starting from the line \(y=(-1)^k\) and ending at the \(x\)-axis with a non-flat step for odd (resp. even) \(k\). 
        Define
        \begin{equation}\label{eq:case1-2}
        \phi_{m,k}(P):=\begin{cases}
        Q\overline{C} & \text{if \(k\) is odd},\\
        \overline{Q}\overline{C} & \text{if \(k\) is even},
        \end{cases}
        \end{equation}
        where
        \[
        Q:=f^{\gamma}\overline{B_1}f^{\alpha}sAf^{\beta}b
        \quad
        \text{and}
        \quad
        C:=\begin{cases}
        \phi_{m',k'}(\overline{B_2}f^{\delta}) & \text{if \(k\) is odd},\\
        \phi_{m',k'}(B_2f^{\delta}) & \text{if \(k\) is even}.
        \end{cases}
        \] 
        Note that \(m'\) is odd in this case.
        The bijection in Case 1 is illustrated in Figure~\ref{fig:case1}.
    \end{enumerate}
\end{enumerate}

\begin{figure}[hbt]
\small{
\begin{subfigure}[b]{0.45\textwidth}
\centering
\begin{tikzpicture}[scale=0.45]

\filldraw[fill=gray!30] (0,-1.5) -- (0,1.5) -- (2,1.5) -- (2,-1.5) --cycle;
\filldraw[fill=gray!30] (4.5,-1.5) -- (4.5,2) -- (6.5,2) -- (6.5,-1.5) --cycle;

\draw[gray, ->] (0,0) -- (8,0);
\draw[gray, ->] (0,-1.5) -- (0,2.5);
\draw[gray] (0,2) -- (7.5,2);

\foreach \i in {1,...,15}
\draw[dotted] (0.5*\i,-1.5) -- (0.5*\i,2)
;	

\foreach \i in {1,...,3}
\draw[dotted] (0,0.5*\i) -- (7.5,0.5*\i)
;	

\foreach \i in {1,...,3}
\draw[dotted] (0,-0.5*\i) -- (7.5,-0.5*\i)
;

\draw[thick] 
(0,0.5) -- (0.5,0.5)
(0.5,1) -- (0,0.5)-- (0.5,0)
(1.5,1) -- (2,1.5) -- (3,1.5) -- (3.5,2) -- (4.5,2) -- (5,1.5)
(6,0.5) -- (6.5,0) -- (7.5,0);

\node at (1,0.75) {\(A\)};
\node[above] at (2.5,1.45) {\(f^{\alpha}\)};
\node[above] at (3.15,1.65) {\(s\)};
\node[above] at (4,1.95) {\(f^{\beta}\)};
\node at (5.5,1) {\(B\)};
\node[above] at (7,-0.05) {\(f^{\gamma}\)};
\node[left] at (0,2) {\(\frac{k+1}{2}\)};
\node[left] at (0,-1.5) {\(-\frac{k-1}{2}\)};
\node[right] at (7.5,2) {\(L_k\)};
\node at (10,0) {};
\node at (-2,0) {};

\end{tikzpicture}\\

\begin{tikzpicture}[scale=0.45]

\filldraw[fill=gray!30] (4.5,-1) -- (4.5,2) -- (6.5,2) -- (6.5,-1) --cycle;
\filldraw[fill=gray!30] (1,-1.5) -- (1,2) -- (3,2) -- (3,-1.5) --cycle;

\draw[gray, ->] (0,-1.5) -- (8,-1.5);
\draw[gray, ->] (0,-1.5) -- (0,2.5);
\draw[gray] (0,2) -- (7.5,2);
\draw[gray] (0,0.5) -- (7.5,0.5);

\foreach \i in {1,...,15}
\draw[dotted] (0.5*\i,-1.5) -- (0.5*\i,2)
;	

\foreach \i in {2,3}
\draw[dotted] (0,0.5*\i) -- (7.5,0.5*\i)
;	

\foreach \i in {0,1,...,3}
\draw[dotted] (0,-0.5*\i) -- (7.5,-0.5*\i)
;

\draw[thick] 
(0,-1.5) -- (1,-1.5) -- (1.5,-1)
(2.5,0) -- (3,0.5) -- (4,0.5) -- (4.5,1) -- (5,1)
(5,0.5) -- (4.5,1) -- (5,1.5)
(6,1.5) -- (6.5,2) -- (7.5,2);

\node[above] at (0.5,-1.55) {\(f^{\gamma}\)};
\node at (2,-0.5) {\(\overline{B}\)};
\node[above] at (3.5,0.45) {\(f^{\alpha}\)};
\node[above] at (4.15,0.65) {\(s\)};
\node at (5.5,1.25) {\(A\)};
\node[above] at (7,1.95) {\(f^{\beta}\)};
\node[left] at (0,2) {\(k\)};
\node[left] at (0,0.5) {\(\frac{k+1}{2}\)};
\node at (10,0) {};
\node at (-2,0) {};
\node[above] at (4,2.5) {\(\updownarrow\)};

\end{tikzpicture}

\subcaption{Case 1--i) for odd \(m\) and \(k\)}
\end{subfigure}
~
\begin{subfigure}[b]{0.45\textwidth}
\centering
\begin{tikzpicture}[scale=0.45]

\filldraw[fill=gray!30] (0,-1.5) -- (0,1.5) -- (2,1.5) -- (2,-1.5) --cycle;
\filldraw[fill=gray!30] (4.5,-1.5) -- (4.5,2) -- (6.5,2) -- (6.5,-1.5) --cycle;
\filldraw[fill=gray!30] (8,-1.5) -- (8,1.5) -- (10,1.5) -- (10,-1.5) --cycle;

\draw[gray, ->] (0,0) -- (11.5,0);
\draw[gray, ->] (0,-1.5) -- (0,2.5);
\draw[gray] (0,2) -- (11,2);

\foreach \i in {1,...,22}
\draw[dotted] (0.5*\i,-1.5) -- (0.5*\i,2)
;	

\foreach \i in {1,...,3}
\draw[dotted] (0,0.5*\i) -- (11,0.5*\i)
;	

\foreach \i in {1,...,3}
\draw[dotted] (0,-0.5*\i) -- (11,-0.5*\i)
;

\draw[thick] 
(0,0.5) -- (0.5,0.5)
(0.5,1) -- (0,0.5)-- (0.5,0)
(1.5,1) -- (2,1.5) -- (3,1.5) -- (3.5,2) -- (4.5,2) -- (5,1.5)
(6,0.5) -- (6.5,0) -- (7.5,0) -- (8,-0.5) -- (8.5, -0.5)
(8.5,-1) -- (8,-0.5) -- (8.5,0)
(9.5,-0.5) -- (10,0) -- (9.5,0.5)
(10,0) -- (11,0);

\node at (1,0.75) {\(A\)};
\node[above] at (2.5,1.45) {\(f^{\alpha}\)};
\node[above] at (3.15,1.65) {\(s\)};
\node[above] at (4,1.95) {\(f^{\beta}\)};
\node at (5.5,1) {\(B_1\)};
\node[above] at (7,-0.05) {\(f^{\delta}\)};
\node[left] at (7.87,-0.27) {\(b\)};
\node at (9,-0.75) {\(B_2\)};
\node[above] at (10.5,-0.05) {\(f^{\gamma}\)};
\node[left] at (0,2) {\(\frac{k+1}{2}\)};
\node[left] at (0,-1.5) {\(-\frac{k-1}{2}\)};
\node[right] at (11,2) {\(L_k\)};
\node at (13,0) {};
\node at (-2,0) {};

\end{tikzpicture}\\

\begin{tikzpicture}[scale=0.45]

\filldraw[fill=gray!30] (4.5,-1) -- (4.5,2) -- (6.5,2) -- (6.5,-1) --cycle;
\filldraw[fill=gray!30] (1,-1.5) -- (1,2) -- (3,2) -- (3,-1.5) --cycle;
\filldraw[fill=gray!30] (8,-1.5) -- (8,1.5) -- (11,1.5) -- (11,-1.5) --cycle;

\draw[gray, ->] (0,-1.5) -- (11.5,-1.5);
\draw[gray, ->] (0,-1.5) -- (0,2.5);
\draw[gray] (0,2) -- (11,2);
\draw[gray] (0,0.5) -- (11,0.5);

\foreach \i in {1,...,22}
\draw[dotted] (0.5*\i,-1.5) -- (0.5*\i,2)
;	

\foreach \i in {2,3}
\draw[dotted] (0,0.5*\i) -- (11,0.5*\i)
;	

\foreach \i in {0,1,...,3}
\draw[dotted] (0,-0.5*\i) -- (11,-0.5*\i)
;

\draw[thick] 
(0,-1.5) -- (1,-1.5) -- (1.5,-1)
(2.5,0) -- (3,0.5) -- (4,0.5) -- (4.5,1) -- (5,1)
(5,1.5) -- (4.5,1) -- (5, 0.5)
(6,1.5) -- (6.5,2) -- (7.5,2) -- (8,1.5) -- (8.5,1)
(8,1.5) -- (8.5,1.5);

\node[above] at (0.5,-1.55) {\(f^{\gamma}\)};
\node at (2,-0.5) {\(\overline{B_1}\)};
\node[above] at (3.5,0.45) {\(f^{\alpha}\)};
\node[above] at (4.15,0.65) {\(s\)};
\node at (5.5,1.25) {\(A\)};
\node[above] at (7,1.95) {\(f^{\beta}\)};
\node[above] at (7.85,1.55) {\(b\)};
\node at (9.5,0) {\(\overline{C}\)};
\node[left] at (0,2) {\(k\)};
\node[left] at (0,0.5) {\(\frac{k+1}{2}\)};
\node at (13,0) {};
\node at (-2,0) {};
\node[above] at (5.75,2.5) {\(\updownarrow\)};

\end{tikzpicture}

\subcaption{Case 1--ii) for odd \(m\) and \(k\)}
\end{subfigure}\\

\vspace{4mm}

\begin{subfigure}[b]{0.45\textwidth}
\centering
\begin{tikzpicture}[scale=0.45]
\node at (-2,0) {};

\filldraw[fill=gray!30] (0,-1.5) -- (0,2) -- (2,2) -- (2,-1.5) --cycle;
\filldraw[fill=gray!30] (4.5,-2) -- (4.5,2) -- (6.5,2) -- (6.5,-2) --cycle;

\draw[gray, ->] (0,0) -- (8,0);
\draw[gray, ->] (0,-2) -- (0,2.5);
\draw[gray] (0,-2) -- (7.5,-2);

\foreach \i in {1,...,15}
\draw[dotted] (0.5*\i,-2) -- (0.5*\i,2)
;	

\foreach \i in {1,...,4}
\draw[dotted] (0,0.5*\i) -- (7.5,0.5*\i)
;	

\foreach \i in {1,...,3}
\draw[dotted] (0,-0.5*\i) -- (7.5,-0.5*\i)
;

\draw[thick] 
(0,0) -- (0.5,0)
(0.5,0.5) -- (0,0)-- (0.5,-0.5)
(1.5,-1) -- (2,-1.5) -- (3,-1.5) -- (3.5,-2) -- (4.5,-2) -- (5,-1.5)
(6,-0.5) -- (6.5,0) -- (7.5,0);

\node at (1,-0.75) {\(A\)};
\node[below] at (2.5,-1.45) {\(f^{\alpha}\)};
\node[below] at (3.15,-1.65) {\(s\)};
\node[below] at (4,-1.95) {\(f^{\beta}\)};
\node at (5.5,-1) {\(B\)};
\node[below] at (7,-0.05) {\(f^{\gamma}\)};
\node[left] at (0,2) {\(\frac{k}{2}\)};
\node[left] at (0,-2) {\(-\frac{k}{2}\)};
\node[right] at (7.5,-2) {\(L_k\)};
\node at (10,0) {};
\node at (-1.5,0) {};

\end{tikzpicture}\\

\begin{tikzpicture}[scale=0.45]

\node at (-2,0) {};
\filldraw[fill=gray!30] (4.5,-1) -- (4.5,2.5) -- (6.5,2.5) -- (6.5,-1) --cycle;
\filldraw[fill=gray!30] (1,-1.5) -- (1,2.5) -- (3,2.5) -- (3,-1.5) --cycle;

\draw[gray, ->] (0,-1.5) -- (8,-1.5);
\draw[gray, ->] (0,-1.5) -- (0,3);
\draw[gray] (0,2.5) -- (7.5,2.5);
\draw[gray] (0,0.5) -- (7.5,0.5);

\foreach \i in {1,...,15}
\draw[dotted] (0.5*\i,-1.5) -- (0.5*\i,2.5)
;	

\foreach \i in {2,3,4}
\draw[dotted] (0,0.5*\i) -- (7.5,0.5*\i)
;	

\foreach \i in {0,1,...,3}
\draw[dotted] (0,-0.5*\i) -- (7.5,-0.5*\i)
;

\draw[thick] 
(0,-1.5) -- (1,-1.5) -- (1.5,-1)
(2.5,0) -- (3,0.5) -- (4,0.5) -- (4.5,1) -- (5,1)
(5,0.5) -- (4.5,1) -- (5,1.5)
(6,2) -- (6.5,2.5) -- (7.5,2.5);

\node[above] at (0.5,-1.55) {\(f^{\gamma}\)};
\node at (2,-0.5) {\(B\)};
\node[above] at (3.5,0.45) {\(f^{\alpha}\)};
\node[above] at (4.15,0.65) {\(\overline{s}\)};
\node at (5.5,1.75) {\(\overline{A}\)};
\node[above] at (7,2.45) {\(f^{\beta}\)};
\node[left] at (0,2.5) {\(k\)};
\node[left] at (0,0.5) {\(\frac{k}{2}\)};
\node at (10,0) {};
\node at (-1.5,0) {};
\node[above] at (4,3) {\(\updownarrow\)};

\end{tikzpicture}

\subcaption{Case 1--i) for even \(m\) and \(k\)}
\end{subfigure}
~
\begin{subfigure}[b]{0.45\textwidth}
\centering
\begin{tikzpicture}[scale=0.45]

\node at (-2,0) {};
\filldraw[fill=gray!30] (0,-1) -- (0,2.5) -- (2,2.5) -- (2,-1) --cycle;
\filldraw[fill=gray!30] (4.5,-1.5) -- (4.5,2.5) -- (6.5,2.5) -- (6.5,-1.5) --cycle;
\filldraw[fill=gray!30] (8,-1) -- (8,2.5) -- (10,2.5) -- (10,-1) --cycle;

\draw[gray, ->] (0,0.5) -- (11.5,0.5);
\draw[gray, ->] (0,-1.5) -- (0,3);
\draw[gray] (0,-1.5) -- (11,-1.5);

\foreach \i in {1,...,22}
\draw[dotted] (0.5*\i,-1.5) -- (0.5*\i,2.5)
;	

\foreach \i in {2,...,5}
\draw[dotted] (0,0.5*\i) -- (11,0.5*\i)
;	

\foreach \i in {0,1,2}
\draw[dotted] (0,-0.5*\i) -- (11,-0.5*\i)
;

\draw[thick] 
(0,0.5) -- (0.5,0.5)
(0.5,1) -- (0,0.5)-- (0.5,0)
(1.5,-0.5) -- (2,-1) -- (3,-1) -- (3.5,-1.5) -- (4.5,-1.5) -- (5,-1)
(6,0) -- (6.5,0.5) -- (7.5,0.5) -- (8,1) -- (8.5, 1)
(8.5,0.5) -- (8,1) -- (8.5,1.5)
(9.5,0) -- (10,0.5) -- (9.5,1)
(10,0.5) -- (11,0.5);

\node at (1,-0.25) {\(A\)};
\node[below] at (2.5,-0.95) {\(f^{\alpha}\)};
\node[below] at (3.15,-1.15) {\(s\)};
\node[below] at (4,-1.45) {\(f^{\beta}\)};
\node at (5.5,-0.5) {\(B_1\)};
\node[below] at (7,0.55) {\(f^{\delta}\)};
\node[left] at (7.95,0.95) {\(b\)};
\node at (9,1) {\(B_2\)};
\node[below] at (10.5,0.45) {\(f^{\gamma}\)};
\node[left] at (0,2.5) {\(\frac{k}{2}\)};
\node[left] at (0,-1.5) {\(-\frac{k}{2}\)};
\node[right] at (11,-1.5) {\(L_k\)};
\node at (13,0) {};
\node at (-1.5,0) {};

\end{tikzpicture}\\

\begin{tikzpicture}[scale=0.45]

\node at (-2,0) {};
\filldraw[fill=gray!30] (4.5,-1) -- (4.5,2.5) -- (6.5,2.5) -- (6.5,-1) --cycle;
\filldraw[fill=gray!30] (1,-1.5) -- (1,2.5) -- (3,2.5) -- (3,-1.5) --cycle;
\filldraw[fill=gray!30] (8,-1.5) -- (8,2) -- (11,2) -- (11,-1.5) --cycle;

\draw[gray, ->] (0,-1.5) -- (11.5,-1.5);
\draw[gray, ->] (0,-1.5) -- (0,3);
\draw[gray] (0,2.5) -- (11,2.5);
\draw[gray] (0,0.5) -- (11,0.5);

\foreach \i in {1,...,22}
\draw[dotted] (0.5*\i,-1.5) -- (0.5*\i,2.5)
;	

\foreach \i in {2,3,4}
\draw[dotted] (0,0.5*\i) -- (11,0.5*\i)
;	

\foreach \i in {0,1,...,3}
\draw[dotted] (0,-0.5*\i) -- (11,-0.5*\i)
;

\draw[thick] 
(0,-1.5) -- (1,-1.5) -- (1.5,-1)
(2.5,0) -- (3,0.5) -- (4,0.5) -- (4.5,1) -- (5,1)
(5,1.5) -- (4.5,1) -- (5, 0.5)
(6,2) -- (6.5,2.5) -- (7.5,2.5) -- (8,2) -- (8.5,1.5)
(8,2) -- (8.5,2);

\node[above] at (0.5,-1.55) {\(f^{\gamma}\)};
\node at (2,-0.5) {\(B_1\)};
\node[above] at (3.5,0.45) {\(f^{\alpha}\)};
\node[above] at (4.15,0.65) {\(\overline{s}\)};
\node at (5.5,1.5) {\(\overline{A}\)};
\node[above] at (7,2.45) {\(f^{\beta}\)};
\node[above] at (7.85,2.05) {\(\overline{b}\)};
\node at (9.5,0.5) {\(\overline{C}\)};
\node[left] at (0,2.5) {\(k\)};
\node[left] at (0,0.5) {\(\frac{k}{2}\)};
\node at (13,0) {};
\node at (-1.5,0) {};
\node[above] at (5.75,3) {\(\updownarrow\)};

\end{tikzpicture}

\subcaption{Case 1--ii) for even \(m\) and \(k\)}
\end{subfigure}
}
\caption{The bijection \(\phi_{m,k}\) in Case 1}
\label{fig:case1}
\end{figure}
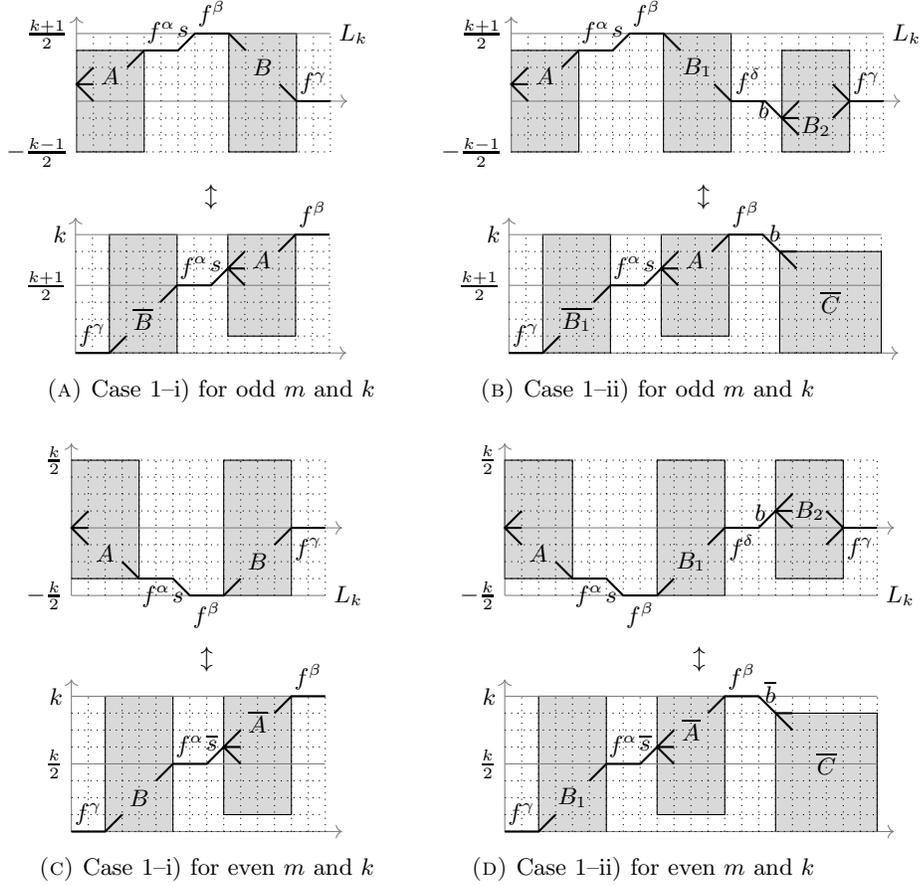

\begin{enumerate}
    \item[\textbf{Case 2}.] Let \(m\) and \(k\) have different parity with \(k>1\). In this case we write \(A\) as \(A_1aA_2\), where \(a\) is the first up (resp. down) step starting from the \(x\)-axis (resp. \(y=1\)) in \(P\) for odd (resp. even) \(k\). Here, \(A_1\) and \(A_2\) can be empty. Note that if \(A_2\) is non-empty, then it never ends with a flat step. Similar to the map in Case 1--ii), we define the map as \eqref{eq:case1-2}, where \(Q\) and \(C\) are given as follows.
    \begin{enumerate}
        \item[{\rm i)}] If there is no break step, then \(P\) can be written as 
        \begin{equation}\label{eq:case2-1P}
            P=A_1aA_2f^{\alpha}sf^{\beta}Bf^{\gamma},
        \end{equation}
        and we set
        \begin{equation}\label{eq:case2-1}
        Q:=f^{\gamma}\overline{B}f^{\alpha}sA_2f^{\beta}\overline{a}
        \quad
        \text{and}
        \quad
        C:=\begin{cases}
        \phi_{m',k'}(A_1) & \text{if \(k\) is odd},\\
        \phi_{m',k'}(\overline{A_1}) & \text{if \(k\) is even}.
        \end{cases}
        \end{equation}
        \item[{\rm ii)}] If there is the break step \(b\), then \(P\) can be written as 
        \begin{equation}\label{eq:case2-2P}
            P=A_1aA_2f^{\alpha}sf^{\beta}B_1f^{\delta}bB_2f^{\gamma},
        \end{equation}
        and we set
        \begin{equation}\label{eq:case2-2}
        Q:=f^{\gamma}\overline{B_1}f^{\alpha}sA_2f^{\beta}\overline{a}
        \quad
        \text{and}
        \quad
        C:=\begin{cases}
        \phi_{m',k'}(A_1\overline{b}\overline{B_2}f^{\delta}) & \text{if \(k\) is odd},\\
        \phi_{m',k'}(\overline{A_1}bB_2f^{\delta}) & \text{if \(k\) is even}.
        \end{cases}
        \end{equation}
    \end{enumerate}
    Note that \(m'\) is even in this case.
    The bijection in Case 2--ii) is illustrated in Figure~\ref{fig:case2}. By regarding \(B_1\) as \(B\) and \(f^{\delta}bB_2\) as \(\emptyset\) in this figure, we see the bijection in Case 2--i). 
\end{enumerate}

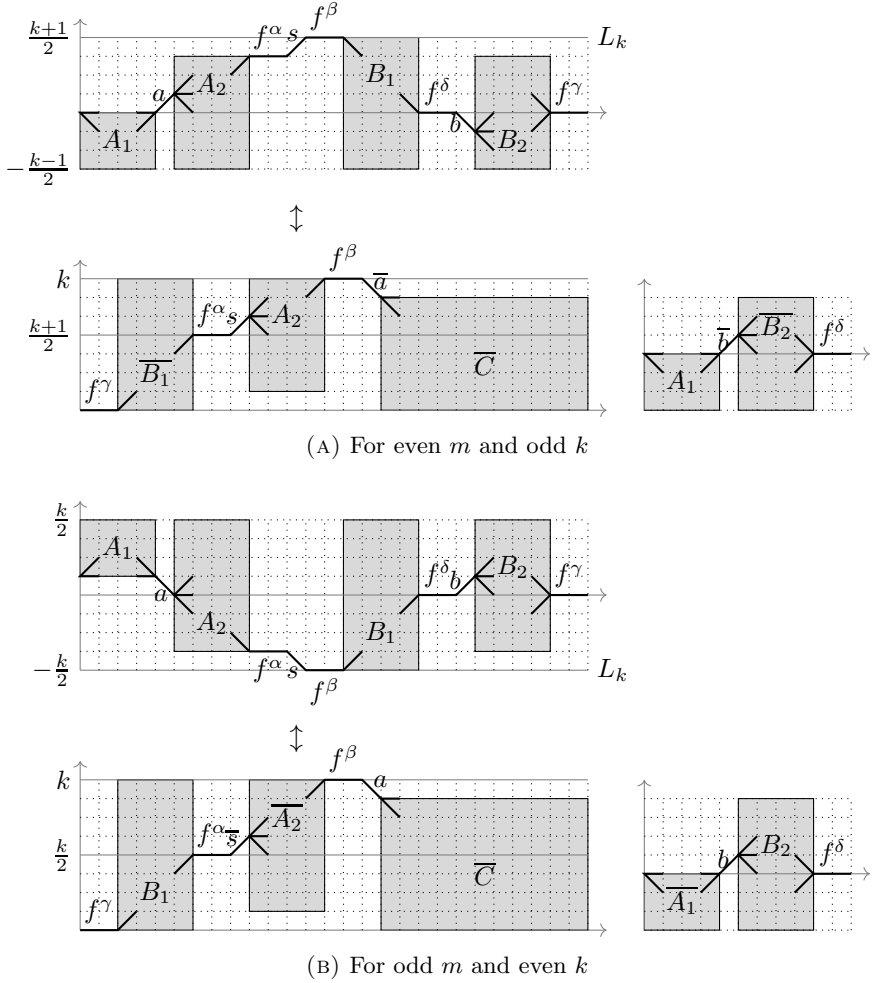
\begin{figure}[hbt]
\centering

\begin{subfigure}[b]{\textwidth}
\centering

\begin{tikzpicture}[scale=0.5]

\filldraw[fill=gray!30] (0,0) -- (0,-1.5) -- (2,-1.5) -- (2,0) --cycle;
\filldraw[fill=gray!30] (2.5,-1.5) -- (2.5,1.5) -- (4.5,1.5) -- (4.5,-1.5) --cycle;
\filldraw[fill=gray!30] (7,-1.5) -- (7,2) -- (9,2) -- (9,-1.5) --cycle;
\filldraw[fill=gray!30] (10.5,-1.5) -- (10.5,1.5) -- (12.5,1.5) -- (12.5,-1.5) --cycle;

\draw[gray, ->] (0,0) -- (14,0);
\draw[gray, ->] (0,-1.5) -- (0,2.5);
\draw[gray] (0,2) -- (13.5,2);

\foreach \i in {1,...,27}
\draw[dotted] (0.5*\i,-1.5) -- (0.5*\i,2)
;	

\foreach \i in {1,...,3}
\draw[dotted] (0,0.5*\i) -- (13.5,0.5*\i)
;	

\foreach \i in {1,...,3}
\draw[dotted] (0,-0.5*\i) -- (13.5,-0.5*\i)
;

\draw[thick] 
(0.5,0) -- (0,0) -- (0.5,-0.5)
(1.5,0) -- (2,0) -- (1.5,-0.5)
(2,0) -- (2.5,0.5) -- (3,0.5)
(3,1) -- (2.5,0.5)-- (3,0)
(4,1) -- (4.5,1.5) -- (5.5,1.5) -- (6,2) -- (7,2) -- (7.5,1.5)
(8.5,0.5) -- (9,0) -- (10,0) -- (10.5,-0.5) -- (11, -0.5)
(11,-1) -- (10.5,-0.5) -- (11,0)
(12,-0.5) -- (12.5,0) -- (12,0.5)
(12.5,0) -- (13.5,0);

\node at (1,-0.75) {\(A_1\)};
\node[above] at (2.1,0.05) {\(a\)};
\node at (3.5,0.75) {\(A_2\)};
\node[above] at (5,1.45) {\(f^{\alpha}\)};
\node[above] at (5.65,1.65) {\(s\)};
\node[above] at (6.5,1.95) {\(f^{\beta}\)};
\node at (8,1) {\(B_1\)};
\node[above] at (9.5,-0.05) {\(f^{\delta}\)};
\node[left] at (10.37,-0.27) {\(b\)};
\node at (11.5,-0.75) {\(B_2\)};
\node[above] at (13,-0.05) {\(f^{\gamma}\)};
\node[left] at (0,2) {\(\frac{k+1}{2}\)};
\node[left] at (0,-1.5) {\(-\frac{k-1}{2}\)};
\node[right] at (13.5,2) {\(L_k\)};
\node at (21.5,0) {};
\node at (-2,0) {};

\end{tikzpicture}\\

\begin{tikzpicture}[scale=0.5]

\filldraw[fill=gray!30] (1,-1.5) -- (1,2) -- (3,2) -- (3,-1.5) --cycle;
\filldraw[fill=gray!30] (4.5,-1) -- (4.5,2) -- (6.5,2) -- (6.5,-1) --cycle;
\filldraw[fill=gray!30] (8,-1.5) -- (8,1.5) -- (13.5,1.5) -- (13.5,-1.5) --cycle;

\filldraw[fill=gray!30] (15,0) -- (15,-1.5) -- (17,-1.5) -- (17,0) --cycle;

\filldraw[fill=gray!30] (17.5,-1.5) -- (17.5,1.5) -- (19.5,1.5) -- (19.5,-1.5) --cycle;

\draw[gray, ->] (0,-1.5) -- (14,-1.5);
\draw[gray, ->] (0,-1.5) -- (0,2.5);
\draw[gray] (0,2) -- (13.5,2);
\draw[gray] (0,0.5) -- (13.5,0.5);

\draw[gray, ->] (15,0) -- (21,0);
\draw[gray, ->] (15,-1.5) -- (15,2);

\foreach \i in {1,...,27}
\draw[dotted] (0.5*\i,-1.5) -- (0.5*\i,2)
;	

\foreach \i in {2,3}
\draw[dotted] (0,0.5*\i) -- (13.5,0.5*\i)
;	

\foreach \i in {0,1,2}
\draw[dotted] (0,-0.5*\i) -- (13.5,-0.5*\i)
;

\foreach \i in {1,...,11}
\draw[dotted] (15+0.5*\i,-1.5) -- (15+0.5*\i,1.5)
;

\foreach \i in {1,2,3}
\draw[dotted] (15,0.5*\i) -- (20.5,0.5*\i)
;	

\foreach \i in {1,2,3}
\draw[dotted] (15,-0.5*\i) -- (20.5,-0.5*\i)
;

\draw[thick] 
(0,-1.5) -- (1,-1.5) -- (1.5,-1) 
(2.5,0) -- (3,0.5) -- (4,0.5) -- (4.5,1) -- (5,1)
(5,1.5) -- (4.5,1) -- (5,0.5)
(6,1.5) -- (6.5,2) -- (7.5,2) -- (8,1.5) -- (8.5,1)
(8,1.5) -- (8.5,1.5)
;

\draw[thick] 
(15.5,0) -- (15,0) -- (15.5,-0.5)
(16.5,0) -- (17,0)
(16.5,-0.5) -- (17,0) -- (17.5,0.5) -- (18,0.5)
(18,0) -- (17.5,0.5) -- (18,1)
(19,0.5) -- (19.5,0) -- (19,-0.5)
(19.5,0) -- (20.5,0)
;

\node[above] at (0.5,-1.55) {\(f^{\gamma}\)};
\node at (2,-0.5) {\(\overline{B_1}\)};
\node[above] at (3.5,0.45) {\(f^{\alpha}\)};
\node[left] at (4.45,0.95) {\(s\)};
\node at (5.5,1) {\(A_2\)};
\node[above] at (7,1.95) {\(f^{\beta}\)};
\node[right] at (7.55,1.95) {\(\overline{a}\)};
\node at (10.75,-0.25) {\(\overline{C}\)};
\node[left] at (0,0.5) {\(\frac{k+1}{2}\)};
\node[left] at (0,2) {\(k\)};
\node at (16,-0.75) {\(A_1\)};
\node at (17.1,0.4) {\(\overline{b}\)};
\node at (18.5,0.75) {\(\overline{B_2}\)};
\node[above] at (20,-0.05) {\(f^{\delta}\)};
\node at (21.5,0) {};
\node at (-2,0) {};

\node[above] at (5.75,3) {\(\updownarrow\)};

\end{tikzpicture}

\subcaption{For even \(m\) and odd \(k\)}
\end{subfigure}

\vspace{4mm}

\begin{subfigure}[b]{\textwidth}
\centering

\begin{tikzpicture}[scale=0.5]

\filldraw[fill=gray!30] (0,0.5) -- (0,2) -- (2,2) -- (2,0.5) --cycle;
\filldraw[fill=gray!30] (2.5,-1.5) -- (2.5,2) -- (4.5,2) -- (4.5,-1.5) --cycle;
\filldraw[fill=gray!30] (7,-2) -- (7,2) -- (9,2) -- (9,-2) --cycle;
\filldraw[fill=gray!30] (10.5,-1.5) -- (10.5,2) -- (12.5,2) -- (12.5,-1.5) --cycle;

\draw[gray, ->] (0,0) -- (14,0);
\draw[gray, ->] (0,-2) -- (0,2.5);
\draw[gray] (0,-2) -- (13.5,-2);

\foreach \i in {1,...,27}
\draw[dotted] (0.5*\i,-2) -- (0.5*\i,2)
;	

\foreach \i in {1,...,4}
\draw[dotted] (0,0.5*\i) -- (13.5,0.5*\i)
;	

\foreach \i in {1,...,3}
\draw[dotted] (0,-0.5*\i) -- (13.5,-0.5*\i)
;

\draw[thick] 
(0.5,0.5) -- (0,0.5) -- (0.5,1)
(1.5,0.5) -- (2,0.5) -- (1.5,1)
(2,0.5) -- (2.5,0) -- (3,0)
(3,0.5) -- (2.5,0)-- (3,-0.5)
(4,-1) -- (4.5,-1.5) -- (5.5,-1.5) -- (6,-2) -- (7,-2) -- (7.5,-1.5)
(8.5,-0.5) -- (9,0) -- (10,0) -- (10.5,0.5) -- (11, 0.5)
(11,0) -- (10.5,0.5) -- (11,1)
(12,-0.5) -- (12.5,0) -- (12,0.5)
(12.5,0) -- (13.5,0);

\node at (1,1.25) {\(A_1\)};
\node[below] at (2.15,0.4) {\(a\)};
\node at (3.5,-0.75) {\(A_2\)};
\node[below] at (5,-1.45) {\(f^{\alpha}\)};
\node[below] at (5.65,-1.65) {\(s\)};
\node[below] at (6.5,-1.95) {\(f^{\beta}\)};
\node at (8,-1) {\(B_1\)};
\node[above] at (9.5,-0.05) {\(f^{\delta}\)};
\node[left] at (10.37,0.4) {\(b\)};
\node at (11.5,0.75) {\(B_2\)};
\node[above] at (13,-0.05) {\(f^{\gamma}\)};
\node[left] at (0,2) {\(\frac{k}{2}\)};
\node[left] at (0,-2) {\(-\frac{k}{2}\)};
\node[right] at (13.5,-2) {\(L_k\)};
\node at (21.5,0) {};
\node at (-2,0) {};

\end{tikzpicture}\\

\begin{tikzpicture}[scale=0.5]

\filldraw[fill=gray!30] (1,-1.5) -- (1,2.5) -- (3,2.5) -- (3,-1.5) --cycle;
\filldraw[fill=gray!30] (4.5,-1) -- (4.5,2.5) -- (6.5,2.5) -- (6.5,-1) --cycle;
\filldraw[fill=gray!30] (8,-1.5) -- (8,2) -- (13.5,2) -- (13.5,-1.5) --cycle;

\filldraw[fill=gray!30] (15,0) -- (15,-1.5) -- (17,-1.5) -- (17,0) --cycle;
\filldraw[fill=gray!30] (17.5,-1.5) -- (17.5,2) -- (19.5,2) -- (19.5,-1.5) --cycle;

\draw[gray, ->] (0,-1.5) -- (14,-1.5);
\draw[gray, ->] (0,-1.5) -- (0,3);
\draw[gray] (0,2.5) -- (13.5,2.5);
\draw[gray] (0,0.5) -- (13.5,0.5);

\draw[gray, ->] (15,0) -- (21,0);
\draw[gray, ->] (15,-1.5) -- (15,2.5);

\foreach \i in {1,...,27}
\draw[dotted] (0.5*\i,-1.5) -- (0.5*\i,2.5)
;	

\foreach \i in {2,3,4}
\draw[dotted] (0,0.5*\i) -- (13.5,0.5*\i)
;	

\foreach \i in {0,1,2}
\draw[dotted] (0,-0.5*\i) -- (13.5,-0.5*\i)
;

\foreach \i in {1,...,11}
\draw[dotted] (15+0.5*\i,-1.5) -- (15+0.5*\i,2)
;

\foreach \i in {1,2,3,4}
\draw[dotted] (15,0.5*\i) -- (20.5,0.5*\i)
;	

\foreach \i in {1,2,3}
\draw[dotted] (15,-0.5*\i) -- (20.5,-0.5*\i)
;

\draw[thick] 
(0,-1.5) -- (1,-1.5) -- (1.5,-1) 
(2.5,0) -- (3,0.5) -- (4,0.5) -- (4.5,1) -- (5,1)
(5,1.5) -- (4.5,1) -- (5,0.5)
(6,2) -- (6.5,2.5) -- (7.5,2.5) -- (8,2) -- (8.5,1.5)
(8,2) -- (8.5,2)
;

\draw[thick] 
(15.5,0) -- (15,0) -- (15.5,-0.5)
(16.5,0) -- (17,0)
(16.5,-0.5) -- (17,0) -- (17.5,0.5) -- (18,0.5)
(18,0) -- (17.5,0.5) -- (18,1)
(19,0.5) -- (19.5,0) -- (19,-0.5)
(19.5,0) -- (20.5,0)
;

\node[above] at (0.5,-1.55) {\(f^{\gamma}\)};
\node at (2,-0.5) {\(B_1\)};
\node[above] at (3.5,0.45) {\(f^{\alpha}\)};
\node[left] at (4.45,0.95) {\(\overline{s}\)};
\node at (5.5,1.5) {\(\overline{A_2}\)};
\node[above] at (7,2.45) {\(f^{\beta}\)};
\node[right] at (7.55,2.45) {\(a\)};
\node at (10.75,0) {\(\overline{C}\)};
\node[left] at (0,0.5) {\(\frac{k}{2}\)};
\node[left] at (0,2.5) {\(k\)};
\node at (16,-0.7) {\(\overline{A_1}\)};
\node at (17.1,0.4) {\(b\)};
\node at (18.5,0.75) {\(B_2\)};
\node[above] at (20,-0.05) {\(f^{\delta}\)};
\node at (21.5,0) {};
\node at (-2,0) {};

\node[above] at (5.75,3) {\(\updownarrow\)};

\end{tikzpicture}

\subcaption{For odd \(m\) and even \(k\)}
\end{subfigure}

\caption{The bijection \(\phi_{m,k}\) in Case 2--ii)}
\label{fig:case2}
\end{figure}

\begin{lem}
For given non-negative integers \(m\) and \(k\), the map \(\phi_{m,k}\) is well-defined.
\end{lem}
\begin{proof}

Let \(P \in \oF(m,r,k)\). In Case 0, 
it is clear that \(\phi_{m,k}(P)=\overleftarrow{P} \in \oM(m,r,k)\). Now consider Case 1--i). 
In this case, for a path \(P\) as in \eqref{eq:basic}, we define \(\phi_{m,k}(P)\) as \eqref{eq:case1-1}. 
If \(k\) is odd (resp. even), then \(A\) is a subpath of \(P\) that starts from the line \(y=1\) (resp. \(x\)-axis), ends on the line \(y=(-1)^{k-1}\lf (k-1)/2 \rf \), and is contained in the strip \(-\lf (k-1)/2 \rf \le y \le \lf k/2 \rf\), while \(B\) is a subpath that starts from the line \(y=(-1)^{k-1}\lf (k+1)/2 \rf\), ends on the \(x\)-axis, and is contained in the strip \(-\lf k/2\rf \le y \le \lf (k+1)/2 \rf\).
Hence, the prefix  \(f^{\gamma}\overline{B}f^{\alpha}\) (resp. \(f^{\gamma}Bf^{\alpha}\)) of \(\phi_{m,k}(P)\) is a Motzkin prefix that ends at the line \(y=\lf (k+1)/2 \rf\) and is contained in the strip \(0\le y \le k\). It follows that the remaining path \(sAf^{\beta}\) (resp. \(\overline{sA}f^{\beta}\)) starts from the line \(y=\lf (k+1)/2 \rf\), ends on the line \(y=k\), and
is contained in the strip \(1\le y \le k\). 
Hence, the prefix \(f^{\gamma}\overline{B}f^{\alpha}\) (resp. \(f^{\gamma}Bf^{\alpha}\)) of \(\phi_{m,k}(P)\) is a Motzkin prefix that ends at the line \(y=\lf (k+1)/2 \rf\) and is contained in the strip \(0\le y \le k\), and the remaining subpath \(sAf^{\beta}\) (resp. \(\overline{sA}f^{\beta}\)) starts from the line \(y=\lf (k+1)/2 \rf\), ends on the line \(y=k\), and
is contained in the strip \(1\le y \le k\) for odd (resp. even) $k$.
Therefore, \(\phi_{m,k}(P) \in \oM(m,r,k)\).

For the remaining cases, we write \(A\) as \(A_1 a A_2\) and \(B\) as \(B_1 f^{\delta}b B_2\) if necessary. Now we use the induction on \(k\). For any \(k'<k\), suppose that \(\phi_{m',k'}(P^*) \in \oM(m',r',k')\) for any paths \(P^* \in \oF(m',r',k')\). Let \(P\in\oF(m,r,k)\), \(\phi_{m,k}(P)\) is defined as \(Q\overline{C}\) or \(\overline{Q}\overline{C}\), where \(Q\) and \(C\) are of the forms in \eqref{eq:case1-2}, \eqref{eq:case2-1}, or \eqref{eq:case2-2}. In any cases, similar to Case 1--i), \(Q\) or \(\overline{Q}\) is a prefix of \(\phi_{m,k}(P)\) that starts from the \(x\)-axis, touches the line \(y=k\), ends on the line \(y=k-1\), and is contained in the strip \(0 \le y \le k\). Since \(C \in \oM(m',r',k')\) with \(k'<k\), \(\overline{C}\) is a suffix of \(\phi_{m,k}(P)\) that starts from the line \(y=k-1\) and is contained in the strip \(k-k'-1 \le y \le k-1\) by the induction hypothesis. Thus, we conclude that \(\phi_{m,k}(P)\in\oM(m,r,k)\).
\end{proof}

\begin{ex}
For given free Motzkin paths, let us apply the map \(\phi_{m,k}\). 
\begin{enumerate}
    \item[(a)] For the path
    \[
    P_1=fduduufdfduufdff\in\overline{\F}(10,6,2),
    \]
    by  applying Case 1--ii) and Case 0, we get 
    \[
    \phi_{10,2}(P_1)=ffuduufdfdufuddf\in\overline{\M}(10,6,2)
    \]
    since \(A=\emptyset\), \(\beta=\delta=0\), and  \(C=\phi_{1,1}(B_2f^{\delta})\) in \eqref{eq:case1-2P}, and \(\phi_{1,1}(fd)=\overleftarrow{fd}=uf\).
    \item[(b)]
    For the path 
    \[
    P_2=fduduufudfdduufudfdffdfufuddfuf\in \overline{\F}(20,11,3),
    \] 
    by applying Case 2-ii), we obtain
    \[
    \phi_{20,3}(P_2)=fufuuddfdufufudffduddfufudfduuf\in\overline{\M}(20,11,3)
    \]
    since \(A_2=\emptyset\), \(\beta=0\), and \(C=\phi_{10,2}(A_1\overline{b}\overline{B_2}f^{\delta})=\phi_{10,2}(P_1)\) in \eqref{eq:case2-2P}, where \(P_1\) is the path given in \emph{(a)}. 
\end{enumerate}
See Figure~\ref{fig:ex} for further details.
\end{ex}

\begin{figure}[hbt]
\small{

\begin{subfigure}[b]{\textwidth}

\begin{tikzpicture}[scale=0.57]
\node at (-1.25,0) {};
\filldraw[fill=gray!30] (1,-0.5) -- (1,0.5) -- (5.5,0.5) -- (5.5,-0.5) --cycle;
\filldraw[fill=gray!30] (6,0) -- (6,0.5) -- (7,0.5) -- (7,0) --cycle;

\draw[gray, ->] (0,0) -- (8.5,0);
\draw[gray, ->] (0,-0.5) -- (0,1);
\draw[gray] (0,-0.5) -- (8,-0.5);

\foreach \i in {1,...,16}
\draw[dotted] (0.5*\i,-0.5) -- (0.5*\i,0.5)
;	

\draw[dotted] (0,0.5) -- (8,0.5);

\draw[thick] 
(0,0) -- (0.5,0) -- (1,-0.5) -- (1.5,0) -- (2,-0.5) -- (2.5,0) -- (3,0.5) -- (3.5,0.5) -- (4,0) -- (4.5,0) -- (5,-0.5) -- (5.5,0) -- (6,0.5) -- (6.5,0.5) -- (7,0) -- (8,0);

\node[below] at (0.25,0.1) {\(f^{\alpha}\)};
\node[below] at (0.65,-0.1) {\(s\)};
\node at (3.25,0) {\(B_1\)};
\node[above] at (5.85,-0.1) {\(b\)};
\node at (6.5,0.2) {\(B_2\)};
\node[above] at (7.5,-0.05) {\(f^{\gamma}\)};
\node[left] at (0,0.5) {\(1\)};
\node[left] at (0,-0.5) {\(-1\)};
\node[right] at (8,-0.5) {\(L_2\)};

\end{tikzpicture}
~
\begin{tikzpicture}[scale=0.57]

\filldraw[fill=gray!30] (0,-0.5) -- (0,0.5) -- (4.5,0.5) -- (4.5,-0.5) --cycle;
\filldraw[fill=gray!30] (6,0) -- (6,-0.5) -- (7,-0.5) -- (7,0) --cycle;

\draw[gray, ->] (-1,-0.5) -- (7.5,-0.5);
\draw[gray, ->] (-1,-0.5) -- (-1,1);
\draw[gray] (-1,0.5) -- (7,0.5);

\foreach \i in {-1,0,1,...,14}
\draw[dotted] (0.5*\i,-0.5) -- (0.5*\i,0.5)
;	

\draw[dotted] (-1,0) -- (7,0);

\draw[thick] 
(-1,-0.5) -- (0,-0.5) -- (0.5,0) -- (1,-0.5) -- (1.5,0) -- (2,0.5) -- (2.5,0.5) -- (3,0) -- (3.5,0) -- (4,-0.5) -- (4.5,0) -- (5,0) -- (5.5,0.5) -- (6,0) -- (6.5,-0.5) -- (7,-0.5);

\node at (-2.2,0) {\(\leftrightarrow\)};
\node[above] at (-0.5,-0.55) {\(f^{\gamma}\)};
\node at (2.25,0) {\(B_1\)};
\node[above] at (4.75,0) {\(f^{\alpha}\)};
\node[above] at (5.2,0.1) {\(\overline{s}\)};
\node[above] at (5.9,0) {\(\overline{b}\)};
\node at (6.5,-0.25) {\(\overline{C}\)};
\node[left] at (-1,0.5) {\(2\)};
\node[left] at (-1,-0.5) {\(0\)};

\end{tikzpicture}

\subcaption{A bijection \(\phi_{10,2}\) in Case 1--ii)}
\end{subfigure}

\vspace{4mm}

\begin{subfigure}[b]{\textwidth}
\centering

\begin{tikzpicture}[scale=0.7]

\filldraw[fill=gray!30] (0,0) -- (0,-0.5) -- (2.5,-0.5) -- (2.5,0) --cycle;
\filldraw[fill=gray!30] (4,-0.5) -- (4,1) -- (9.5,1) -- (9.5,-0.5) --cycle;
\filldraw[fill=gray!30] (11,-0.5) -- (11,0.5) -- (15,0.5) -- (15,-0.5) --cycle;

\draw[gray, ->] (0,0) -- (16,0);
\draw[gray, ->] (0,-0.5) -- (0,1.5);
\draw[gray] (0,1) -- (15.5,1);

\foreach \i in {1,...,31}
\draw[dotted] (0.5*\i,-0.5) -- (0.5*\i,1)
;	

\draw[dotted] (0,0.5) -- (15.5,0.5);	

\draw[dotted] (0,-0.5) -- (15.5,-0.5);

\draw[thick] 
(0,0) -- (0.5, 0) -- (1,-0.5) -- (1.5,0) -- (2,-0.5) -- (2.5,0) -- (3,0.5) -- (3.5,0.5) -- (4,1) -- (4.5,0.5) -- (5,0.5) -- (5.5,0) -- (6,-0.5) -- (6.5,0) -- (7,0.5) -- (7.5,0.5) -- (8,1) -- (8.5,0.5) -- (9,0.5) -- (9.5,0) -- (10,0) -- (10.5,0) -- (11,-0.5) -- (11.5,-0.5) -- (12,0) -- (12.5,0) -- (13,0.5) -- (13.5,0) -- (14,-0.5) -- (14.5,-0.5) -- (15,0) -- (15.5,0);

\node at (1.25,-0.26) {\(A_1\)};
\node[above] at (2.6,0.1) {\(a\)};
\node[above] at (3.25,0.4) {\(f^{\alpha}\)};
\node[above] at (3.6,0.55) {\(s\)};
\node at (6.75,-0.25) {\(B_1\)};
\node[above] at (10,-0.05) {\(f^{\delta}\)};
\node[left] at (10.87,-0.28) {\(b\)};
\node at (13,-0.25) {\(B_2\)};
\node[above] at (15.25,-0.05) {\(f^{\gamma}\)};
\node[left] at (0,1) {\(2\)};
\node[left] at (0,-0.5) {\(-1\)};
\node[right] at (15.5,1) {\(L_3\)};
\node at (17,0) {};
\node at (-1,0) {};

\end{tikzpicture}\\[2ex]

\begin{tikzpicture}[scale=0.7]

\filldraw[fill=gray!30] (0.5,-0.5) -- (0.5,1) -- (6,1) -- (6,-0.5) --cycle;
\filldraw[fill=gray!30] (7.5,-0.5) -- (7.5,0.5) -- (15.5,0.5) -- (15.5,-0.5) --cycle;

\draw[gray, ->] (0,-0.5) -- (16,-0.5);
\draw[gray, ->] (0,-0.5) -- (0,1.5);
\draw[gray] (0,1) -- (15.5,1);

\foreach \i in {1,...,31}
\draw[dotted] (0.5*\i,-0.5) -- (0.5*\i,1)
;	

\draw[dotted] (0,0.5) -- (15.5,0.5);	

\draw[dotted] (0,0) -- (15.5,0);

\draw[thick] 
(0,-0.5) -- (0.5, -0.5) -- (1,0) -- (1.5,0) -- (2,0.5) -- (2.5,1) -- (3,0.5) -- (3.5,0) -- (4,0) -- (4.5,-0.5) -- (5,0) -- (5.5,0) -- (6,0.5) -- (6.5,0.5) -- (7,1) -- (7.5,0.5) -- (8,0.5) -- (8.5,0.5) -- (9,0) -- (9.5,0.5) -- (10,0) -- (10.5,-0.5) -- (11,-0.5) -- (11.5,0) -- (12,0) -- (12.5,0.5) -- (13,0) -- (13.5,0) -- (14,-0.5) -- (14.5,0) -- (15,0.5) -- (15.5,0.5);

\node at (8,2) {\(\updownarrow\)};
\node[above] at (0.25,-0.55) {\(f^{\gamma}\)};
\node at (3.25,-0.25) {\(\overline{B_1}\)};
\node[above] at (6.25,0.4) {\(f^{\alpha}\)};
\node[above] at (6.6,0.55) {\(s\)};
\node[above] at (7.4,0.54) {\(\overline{a}\)};
\node at (11.5,-0.25) {\(\overline{C}\)};
\node[left] at (0,1) {\(3\)};
\node[left] at (0,-0.5) {\(0\)};

\node at (17,0) {};
\node at (-1,0) {};

\end{tikzpicture}

\subcaption{A bijection \(\phi_{20,3}\) in Case 2--ii)}
\end{subfigure}
}
\caption{Examples of the map \(\phi_{m,k}\)}
\label{fig:ex}
\end{figure}
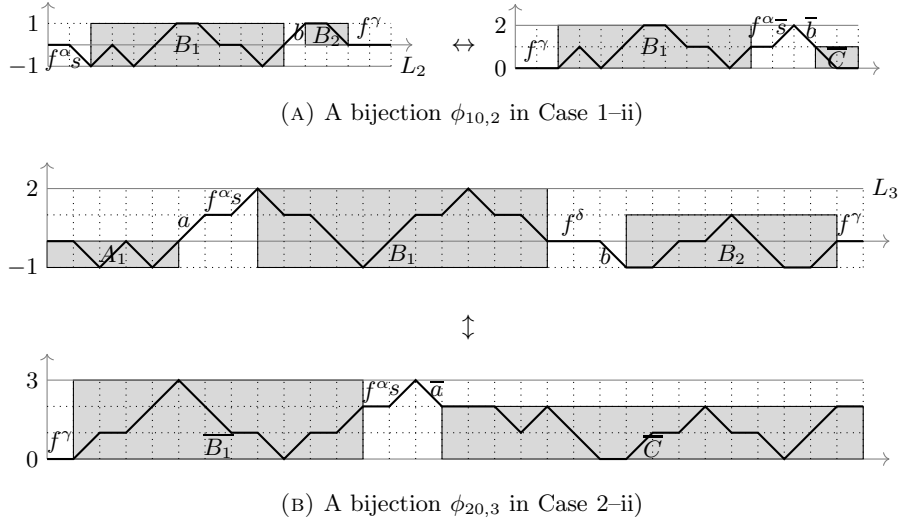


\subsection{Map \(\psi_{m,k}\)}\label{sec:1-1}

Now we define a map 
\[
\psi_{m,k}: \bigcup_{r\geq 0}\oM(m,r,k) \rightarrow \bigcup_{r\geq 0}\oF(m,r,k)
\] 
and show that \(\psi_{m,k}=\phi_{m,k}^{-1}\).
Let $S$ be a path in the set $\oM(m,r,k)$ for some $r\ge 0$.

\begin{enumerate}
    \item[\textbf{Case 0}.]
    For $k=0$ or $1$, we define $\psi_{m,k}(S)=\overleftarrow{S}$.
\end{enumerate}

Recall that the last vertex on the line $y=k$ is called the turning point of a path in \(\oM(m,r,k)\).
We define a \emph{critical point} of $S$ as the rightmost point on the $x$-axis which locates before the turning point.
\begin{enumerate}
    \item[\textbf{Case I}.]     
    For $k>1$, assume that $S$ is a path which ends on the line $y=k$. Note that $m$ and $k$ have the same parity and we write
    \[
    S=f^{\gamma}B^* f^{\alpha}u^*A^* f^{\beta},
    \]
    where $u^*$ is the first up step starting from the line $y=\lf (k+1)/2 \rf$ after the critical point of $S$, $\gamma\ge 0$ (resp. $\beta \ge 0$) is the maximum number of consecutive initial (resp. final) flat steps of $P$, and $\alpha \ge 0$ is the maximum number of consecutive flat steps before the step $u^*$. Hence, $B^*$ (resp. $A^*$) is the subpath of $S$ such that it starts from the $x$-axis (resp. $y=\lf k/2 \rf +1$), ends on the line $y=\lf k/2 \rf$ (resp. $y=k$), and is contained in the strip $0\le y \le k$ (resp. $1\le y \le k$). We define
    \begin{equation}\label{eq:case1-1i}
        \psi_{m,k}(S):=\begin{cases}
        A^* f^{\alpha}u^* f^{\beta}\overline{B^*}f^{\gamma} & \text{if \(k\) is odd},\\
        \overline{A^*}f^{\alpha}\overline{u^*}f^{\beta}B^* f^{\gamma} & \text{if \(k\) is even}.
        \end{cases}
    \end{equation}
    
    \item[\textbf{Case II}.] Suppose that $S$ is a path which does not end on the line $y=k$ for $k>1$. In this case, we write
    \[
    S=f^{\gamma}B^* f^{\alpha}u^*A^* f^{\beta}d^*\overline{C^*},
    \]
    where $u^*,A^*,B^*,\alpha,\beta,\gamma$ is defined as in Case I, $d^*$ is the last down step starting from the line $y=k$, and $\overline{C^*}$ is a suffix of $S$ after the step $d^*$. Note that $C^*\in \oM(m',r',k')$ for some $k'<k$ since $C^*$ is contained in the strip $0 \le y \le k-1$.

    \begin{enumerate}
        \item[{\rm i)}] Let $m$ and $k$ have the same parity, which follows that $m'$ is odd. We write $\psi_{m',k'}(C^*)=B^{\bullet}f^{\delta}$, where $\delta\ge 0$ is the maximum number of consecutive flat steps at the suffix of $\psi_{m',k'}(C^*)$. We set
        \begin{equation}\label{eq:case1-2i}
            \psi_{m,k}(S):=\begin{cases}
            A^* f^{\alpha}u^* f^{\beta}\overline{B^*}f^{\delta}d^*\overline{B^{\bullet}}f^{\gamma} & \text{if \(k\) is odd},\\
            \overline{A^*}f^{\alpha}\overline{u^*}f^{\beta}B^*f^{\delta}\overline{d^*}B^{\bullet} f^{\gamma} & \text{if \(k\) is even}.
            \end{cases}
        \end{equation}

        \item[{\rm ii)}] Let $m$ and $k$ have different parity. In this case, $m'$ is even. We divide two cases whether $\psi_{m',k'}(C^*)$ goes above the $x$-axis or not.

        If $\psi_{m',k'}(C^*)$ does not go above the $x$-axis, then we write $\psi_{m',k'}(C^*)=A^{\bullet}$ and define
        \begin{equation}\label{eq:case2-1i}
            \psi_{m,k}(S):=\begin{cases}
            A^{\bullet}\overline{d^*} A^* f^{\alpha}u^* f^{\beta}\overline{B^*}f^{\gamma} & \text{if \(k\) is odd},\\
            \overline{A^{\bullet}}d^*\overline{A^*}f^{\alpha}\overline{u^*}f^{\beta}B^*f^{\gamma} & \text{if \(k\) is even}.
            \end{cases}
        \end{equation}

        If $\psi_{m',k'}(C^*)$ goes above the $x$-axis, then we write $\psi_{m',k'}(C^*)=A^{\bullet}u^{\bullet}B^{\bullet}f^{\delta}$, where $u^{\bullet}$ is the first up step starting from the $x$-axis and $\delta\ge 0$ is the maximum number of consecutive flat steps at the suffix of $\psi_{m',k'}(C^*)$. We define
        \begin{equation}\label{eq:case2-2i}
            \psi_{m,k}(S):=\begin{cases}
            A^{\bullet}\overline{d^*} A^* f^{\alpha}u^* f^{\beta}\overline{B^*}f^{\delta}\overline{u^{\bullet}}\overline{B^{\bullet}}f^{\gamma} & \text{if \(k\) is odd},\\
            \overline{A^{\bullet}}d^*\overline{A^*}f^{\alpha}\overline{u^*}f^{\beta}B^* f^{\delta}u^{\bullet}B^{\bullet}f^{\gamma} & \text{if \(k\) is even}.
            \end{cases}
        \end{equation}
        \end{enumerate}
\end{enumerate}

\begin{lem}\label{lem:inverse}
The map \(\psi_{m,k}\) is the inverse map of \(\phi_{m,k}\).
\end{lem}

\begin{proof}
For $k=0$ or $1$, it is clear that $\psi_{m,k}(\phi_{m,k}(P))=P$ for any path $P \in \oF(m,r,k)$ by the construction.

From now on, we set $k>1$.
Let $P \in \oF(m,r,k)$ when $m$ and $k$ have the same parity and there is no break step in $P$ so that $P$ is represented as $P=Af^{\alpha}sf^{\beta}Bf^{\gamma}$. When $k$ is odd (resp. even), it follows from \eqref{eq:case1-1} and \eqref{eq:case1-1i} that $\psi_{m,k}(\phi_{m,k}(P))=P$ since $A=A^*$ (resp. $A=\overline{A^*}$), $s=u^*$ (resp. $s=\overline{u^*}$), and $B=\overline{B^*}$ (resp. $B=B^*$), where $S=\phi_{m,k}(P)\in \oM(m,r,k)$.

Now, we assume that $\psi_{m',k'}(P^*) \in \oF(m',r',k')$ for any path $P^*$ with $k'<k$.
Let $P \in \oF(m,r,k)$ when $m$ and $k$ have the same parity and there is a break step $b$ in $P$ so that $P$ is represented as $P=Af^{\alpha}sf^{\beta}B_1 f^{\delta}bB_2 f^{\gamma}$. For odd (resp. even) $k$, according to \eqref{eq:case1-2} and \eqref{eq:case1-2i}, $\psi_{m,k}(\phi_{m,k}(P))=P$ since  
$A=A^*$ (resp. $A=\overline{A^*}$), $s=u^*$ (resp. $s=\overline{u^*}$), $B_1=\overline{B^*}$ (resp. $B_1=B^*$), $b=d^*$ (resp. $b=\overline{d^*}$), and $B_2=\overline{B^\bullet}$ (resp. $B_2=B^\bullet$).

Similarly, by \eqref{eq:case2-1}, \eqref{eq:case2-2}, \eqref{eq:case2-1i}, and \eqref{eq:case2-2i}, we can see that $\psi_{m,k}(\phi_{m,k}(P))=P$, where $P \in \oF(m,r,k)$ when $m$ and $k$ have different parity with $k>1$.
\end{proof}

\begin{ex}
For given Motzkin prefixes, 
\begin{align*}
S_1&=uufufdddufuuf \in \overline{\M}(9,4,3),\\
S_2&=uufuufdddfdfuuudfddf \in \overline{\M}(14,6,4),\\
S_3&=fuuufuduufuffdddfddfuuufufdddufuuf \in \overline{\M}(23,11,6),
\end{align*}
we have
\begin{align*}
\psi_{9,3}(S_1)&=ufddfdfuuudfd \in \overline{\F}(9,4,3),\\
\psi_{14,4}(S_2)&=dfdfuuuufddfdfuuudfd \in \overline{\F}(14,6,4),\\
\psi_{23,6}(S_3)&=ufufddduddfdfdffuuuuufddfdfuuudfdf \in \overline{\F}(23,11,6).
\end{align*}
See Figure~\ref{fig:ex2} for further details.
\end{ex}

\begin{figure}[hbt]
\small{

\begin{subfigure}[b]{0.35\textwidth}

\begin{tikzpicture}[scale=0.6]

\filldraw[fill=gray!30] (0,0) -- (0,1.5) -- (5.5,1.5) -- (5.5,0) --cycle;

\draw[gray, ->] (0,0) -- (7,0);
\draw[gray, ->] (0,0) -- (0,2);
\draw[gray] (0,1.5) -- (6.5,1.5);

\foreach \i in {1,...,13}
\draw[dotted] (0.5*\i,0) -- (0.5*\i,1.5)
;	

\draw[dotted] (0,0.5) -- (6.5,0.5);
\draw[dotted] (0,1) -- (6.5,1);

\draw[thick] 
(0,0) -- (0.5,0.5) -- (1,1) -- (1.5,1) -- (2,1.5) -- (2.5,1.5) -- (3,1) -- (3.5,0.5) -- (4,0) -- (4.5,0.5) -- (5,0.5) -- (5.5,1) -- (6,1.5) -- (6.5,1.5);

\node at (2.75,0.75) {\(B^*\)};
\node[above] at (5.9,0.8) {\(u^*\)};
\node[above] at (6.27,0.7) {\(f^{\beta}\)};
\node[left] at (0,1.5) {\(3\)};
\node[left] at (0,0) {\(0\)};

\node at (-1,2.1) {};
\end{tikzpicture}\\

\begin{tikzpicture}[scale=0.6]

\filldraw[fill=gray!30] (1,-0.5) -- (1,1) -- (6.5,1) -- (6.5,-0.5) --cycle;

\node at (-1,-1.2) {};
\draw[gray, ->] (0,0) -- (7,0);
\draw[gray, ->] (0,-0.5) -- (0,1.5);
\draw[gray] (0,1) -- (6.5,1);

\foreach \i in {1,...,13}
\draw[dotted] (0.5*\i,-0.5) -- (0.5*\i,1)
;	

\draw[dotted] (0,0.5) -- (6.5,0.5);
\draw[dotted] (0,-0.5) -- (6.5,-0.5);

\draw[thick] 
(0,0.5) -- (0.5,1) -- (1,1) -- (1.5,0.5) -- (2,0) -- (2.5,0) -- (3,-0.5) -- (3.5,-0.5) -- (4,0) -- (4.5,0.5) -- (5,1) -- (5.5,0.5) -- (6,0.5) -- (6.5,0);

\node at (3.25,2.2) {\(\updownarrow\)};
\node[below] at (0.4,1) {\(u^*\)};
\node[below] at (0.75,1.05) {\(f^{\beta}\)};
\node at (3.75,0.25) {\(\overline{B^*}\)};
\node[left] at (0,1) {\(2\)};
\node[left] at (0,-0.5) {\(-1\)};
\node[right] at (6.5,1) {\(L_3\)};

\end{tikzpicture}

\subcaption{A bijection \(\psi_{9,3}\) in Case I}
\end{subfigure}
~~~
\begin{subfigure}[b]{0.65\textwidth}
\centering

\begin{tikzpicture}[scale=0.6]

\filldraw[fill=gray!30] (0,0) -- (0,2) -- (1,2) -- (1,0) --cycle;
\filldraw[fill=gray!30] (2,0.5) -- (2,2) -- (2.5,2) -- (2.5,0.5) --cycle;
\filldraw[fill=gray!30] (3.5,0) -- (3.5,1.5) -- (10,1.5) -- (10,0) --cycle;

\draw[gray, ->] (0,0) -- (10.5,0);
\draw[gray, ->] (0,0) -- (0,2.5);
\draw[gray] (0,2) -- (10,2);

\foreach \i in {1,...,20}
\draw[dotted] (0.5*\i,0) -- (0.5*\i,2);	

\foreach \i in {1,2,3}
\draw[dotted] (0,0.5*\i) -- (10,0.5*\i);

\draw[thick] 
(0,0) -- (0.5, 0.5) -- (1,1) -- (1.5,1) -- (2,1.5) -- (2.5,2) -- (3,2) -- (3.5,1.5) -- (4,1) -- (4.5,0.5) -- (5,0.5) -- (5.5,0) -- (6,0) -- (6.5,0.5) -- (7,1) -- (7.5,1.5) -- (8,1) -- (8.5,1) -- (9,0.5) -- (9.5,0) -- (10,0);

\node at (0.5,1) {\(B^*\)};
\node[above] at (1.25,0.95) {\(f^{\alpha}\)};
\node[above] at (1.64,1.05) {\(u^*\)};
\node at (2.25,1.25) {\(A^*\)};
\node[above] at (2.75,1.95) {\(f^{\beta}\)};
\node[above] at (3.42,1.65) {\(d^*\)};
\node at (6.75,0.75) {\(\overline{C^*}\)};
\node[left] at (0,2) {\(4\)};
\node[left] at (0,0) {\(0\)};

\end{tikzpicture}\\

\begin{tikzpicture}[scale=0.6]

\filldraw[fill=gray!30] (0,1) -- (0,-0.5) -- (0.5,-0.5) -- (0.5,1) --cycle;
\filldraw[fill=gray!30] (2,-1) -- (2,1) -- (3,1) -- (3,-1) --cycle;
\filldraw[fill=gray!30] (3.5,-0.5) -- (3.5,1) -- (10,1) -- (10,-0.5) --cycle;

\draw[gray, ->] (0,0) -- (10.5,0);
\draw[gray, ->] (0,-1) -- (0,1.5);
\draw[gray] (0,-1) -- (10,-1);

\foreach \i in {1,...,20}
\draw[dotted] (0.5*\i,-1) -- (0.5*\i,1)
;	

\foreach \i in {-1,1,2}
\draw[dotted] (0,0.5*\i) -- (10,0.5*\i);	

\draw[thick] 
(0,0) -- (0.5, -0.5) -- (1,-0.5) -- (1.5,-1) -- (2,-1) -- (2.5,-0.5) -- (3,0) -- (3.5,0.5) -- (4,1) -- (4.5,1) -- (5,0.5) -- (5.5,0) -- (6,0) -- (6.5,-0.5) -- (7,-0.5) -- (7.5,0) -- (8,0.5) -- (8.5,1) -- (9,0.5) -- (9.5,0.5) -- (10,0);

\node at (5.25,2) {\(\updownarrow\)};
\node at (0.25,0.25) {\(\overline{A^*}\)};
\node[below] at (0.65,-0.4) {\(f^{\alpha}\)};
\node[below] at (1.05,-0.4) {\(\overline{u^*}\)};
\node[below] at (1.75,-0.95) {\(f^{\beta}\)};
\node at (2.5,0) {\(B^*\)};
\node[above] at (3.2,0) {\(\overline{d^*}\)};
\node at (6.75,0.25) {\(B^{\bullet}\)};
\node[left] at (0,1) {\(2\)};
\node[left] at (0,-1) {\(-2\)};
\node[right] at (10,-1) {\(L_4\)};

\end{tikzpicture}

\subcaption{A bijection \(\psi_{14,4}\) in Case II--i)}
\end{subfigure}

\vspace{4mm}

\begin{subfigure}[b]{\textwidth}
\centering

\begin{tikzpicture}[scale=0.6]

\filldraw[fill=gray!30] (0.5,0) -- (0.5,3) -- (2,3) -- (2,0) --cycle;
\filldraw[fill=gray!30] (3,0.5) -- (3,3) -- (5.5,3) -- (5.5,0.5) --cycle;
\filldraw[fill=gray!30] (7,0) -- (7,2.5) -- (17,2.5) -- (17,0) --cycle;

\draw[gray, ->] (0,0) -- (17.5,0);
\draw[gray, ->] (0,0) -- (0,3.5);
\draw[gray] (0,3) -- (17,3);

\foreach \i in {1,...,34}
\draw[dotted] (0.5*\i,0) -- (0.5*\i,3);	

\foreach \i in {1,...,5}
\draw[dotted] (0,0.5*\i) -- (17,0.5*\i);	

\draw[thick] 
(0,0) -- (0.5,0) -- (1,0.5) -- (1.5,1) -- (2,1.5) -- (2.5,1.5) -- (3,2) -- (3.5,1.5) -- (4,2) -- (4.5,2.5) -- (5,2.5) -- (5.5,3) -- (6,3) -- (6.5,3) -- (7,2.5) -- (7.5,2) -- (8,1.5) -- (8.5,1.5) -- (9,1) -- (9.5,0.5) -- (10,0.5) -- (10.5,1) -- (11,1.5) -- (11.5,2) -- (12,2) -- (12.5,2.5) -- (13,2.5) -- (13.5,2) -- (14,1.5) -- (14.5,1) -- (15,1.5) -- (15.5,1.5) -- (16,2) -- (16.5,2.5) -- (17,2.5);

\node[above] at (0.25,-0.05) {\(f^{\gamma}\)};
\node at (1.25,1.5) {\(B^*\)};
\node[above] at (2.25,1.45) {\(f^{\alpha}\)};
\node[above] at (2.62,1.55) {\(u*\)};
\node at (4.25,1.75) {\(A^*\)};
\node[above] at (6,2.95) {\(f^{\beta}\)};
\node[above] at (6.9,2.6) {\(d^*\)};
\node at (12,1.25) {\(\overline{C^*}\)};
\node[left] at (0,3) {\(6\)};
\node[left] at (0,0) {\(0\)};

\node at (18,0) {};
\node at (-1,0) {};

\end{tikzpicture}\\[2ex]

\begin{tikzpicture}[scale=0.6]

\filldraw[fill=gray!30] (0,0.5) -- (0,1.5) -- (3,1.5) -- (3,0.5) --cycle;
\filldraw[fill=gray!30] (3.5,1.5) -- (3.5,-1) -- (6,-1) -- (6,1.5) --cycle;
\filldraw[fill=gray!30] (8,-1.5) -- (8,1.5) -- (9.5,1.5) -- (9.5,-1.5) --cycle;
\filldraw[fill=gray!30] (10,-1) -- (10,1.5) -- (16.5,1.5) -- (16.5,-1) --cycle;

\draw[gray, ->] (0,0) -- (17.5,0);
\draw[gray, ->] (0,-1.5) -- (0,2);
\draw[gray] (0,-1.5) -- (17,-1.5);

\foreach \i in {1,...,34}
\draw[dotted] (0.5*\i,-1.5) -- (0.5*\i,1.5)
;	

\foreach \i in {-2,-1,1,2,3}
\draw[dotted] (0,0.5*\i) -- (17,0.5*\i);	

\draw[thick] 
(0,0.5) -- (0.5,1) -- (1,1) -- (1.5,1.5) -- (2,1.5) -- (2.5,1) -- (3,0.5) -- (3.5,0) -- (4,0.5) -- (4.5,0) -- (5,-0.5) -- (5.5,-0.5) -- (6,-1) -- (6.5,-1) -- (7,-1.5) -- (7.5,-1.5) -- (8,-1.5) -- (8.5,-1) -- (9,-0.5) -- (9.5,0) -- (10,0.5) -- (10.5,1) -- (11,1) -- (11.5,0.5) -- (12,0) -- (12.5,0) -- (13,-0.5) -- (13.5,-0.5) -- (14,0) -- (14.5,0.5) -- (15,1) -- (15.5,0.5) -- (16,0.5) -- (16.5,0) -- (17,0);

\node at (8.75,2.5) {\(\updownarrow\)};
\node at (1.5,1) {\(\overline{A^{\bullet}}\)};
\node[above] at (3.3,0.2) {\(d^*\)};
\node at (4.75,0.25) {\(\overline{A^*}\)};
\node[below] at (6.15,-0.9) {\(f^{\alpha}\)};
\node[below] at (6.55,-0.9) {\(\overline{u^*}\)};
\node[below] at (7.5,-1.45) {\(f^{\beta}\)};
\node at (8.75,0) {\(B^*\)};
\node at (9.75,0.4) {\(u^{\bullet}\)};
\node at (13.25,0.25) {\(B^{\bullet}\)};
\node[above] at (16.75,-0.05) {\(f^{\gamma}\)};
\node[left] at (0,1.5) {\(3\)};
\node[left] at (0,-1.5) {\(-3\)};
\node[right] at (17,-1.5) {\(L_6\)};
\node at (18,0) {};
\node at (-1,0) {};

\end{tikzpicture}

\subcaption{A bijection \(\psi_{23,6}\) in Case II--ii)}
\end{subfigure}
}
\caption{Examples of the map \(\psi_{m,k}\)}
\label{fig:ex2}
\end{figure}
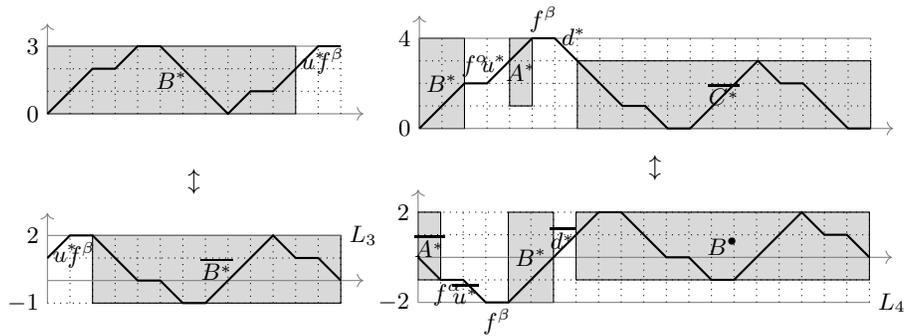
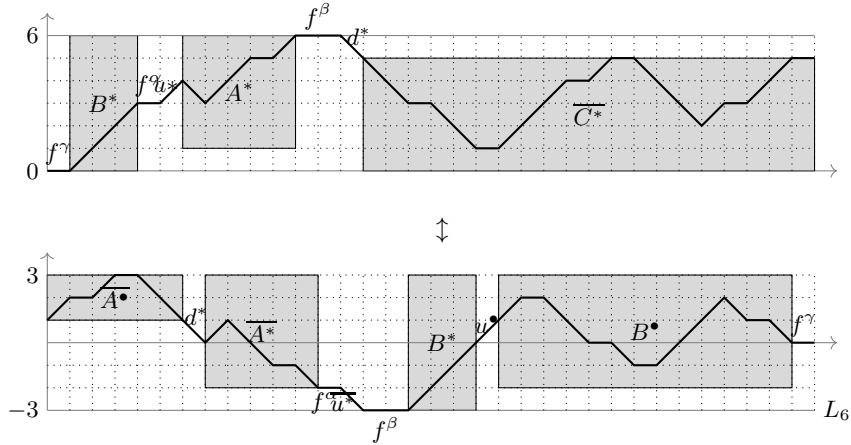

\subsection{A property of \(\phi_{m,k}\)}\label{sec:special}

For a free Motzkin path \(P\),
a maximal subpath in \(P\) with no down (resp. up) step is called an \emph{upward} (resp. \emph{downward}) \emph{run} if it contains at least one up (resp. down) step. Let \(\run(P)\) denote the total number of runs in \(P\). If \(P\) has no flat step, then the total number of peaks and valleys of $P$ is counted by \(\run(P)-1\). 
For example, the path \(P=uufufdddufuuf\) has two upward runs, \(uufuf\) and \(ufuuf\), and one downward run \(fddd\) so that \(\run(P)=3\). Note that \(\run(P)=0\) if and only if \(P\) is empty or a path consisting of flat steps only, and \(\run(P)=\run(\overline{P})\).
For a path $P \in \oF(m,r,k)$, the following proposition shows that $\run(P)$ and $\run(\phi_{m,k}(P))$ are differ by at most $1$. 

\begin{prop}\label{prop:run}
For positive integers \(m\) and \(k\), let \(P\in \oF(m,r,k)\) be given.

\begin{enumerate}
    \item[(a)] If \(P\) starts with an upward run, then 
    \[
    \run(\phi_{m,k}(P))=\run(P)-\{1-(-1)^m\}/2.
    \]
    \item[(b)] If \(P\) starts with a downward run, then 
    \[
    \run(\phi_{m,k}(P))=\run(P)-\{1+(-1)^m\}/2.
    \]
\end{enumerate}
\end{prop}

\begin{proof}
As erasing any number of flat steps do not change the number of runs, it suffices to show that this proposition holds when \(r=0\). We prove it by using induction on \(k\). 

For the initial step with \(k=1\), we consider Case 0. Recall that $\phi_{m,1}(P)=\overleftarrow{P}$. If \(P\) starts with an up step, \(m\) must be even and \(\run(\overleftarrow{P})=\run(P)\). When $P$ starts with a down step, $m$ is odd and \(\run(\overleftarrow{P})=\run(P)\).

Now we assume $k>1$ and suppose that this proposition holds for any \(P^*\in \oF(m',0,k')\) with \(k'<k\). Here we give a detailed proof for (a) and the proof for (b) comes out similarly.

Suppose that \(P\in\oF(m,0,k)\) starts with an up step and let \(S=\phi_{m,k}(P)\). We need to show that \(\run(S)\) is given by \(\run(P)-1\) (resp. \(\run(P)\)) if \(m\) is odd (resp. even).

In Case 1--i), we write \(P=AuB\) (resp. $P=AdB$) and \(S=\overline{B}uA\) (resp. \(S=Bu\overline{A}\)), where \(\overline{B}\) (resp. $B$) ends with an up step if \(m\) is odd (resp. even). If \(A\) is empty, then \(m\) must be odd so that \(\run(S)=\run(B)=\run(P)-1\) as we desire. Now assume that \(A\) starts with an up step. In this case,
\(\run(P)=\run(A)+\run(B)\) and 
\[
\run(S)=\begin{cases}
\run(A)+\run(B)-1 & \mbox{if \(m\) is odd},\\
\run(A)+\run(B) & \mbox{if \(m\) is even},
\end{cases}
\]
so we are done.

In Case 1--ii), we write \(P=AuB_1dB_2\) (resp. $P=AdB_1 u B_2$) and \(\overline{B_2}\in\oF(m',k')\) (resp. \(B_2\in\oF(m',k')\)) for some $k'<k$ and odd \(m'\), where $m$ is odd (resp. even). Let \(r:=\run(A)+\run(B_1)\). Note that if \(m\) is odd (resp. even), then
\[
\run(P)=\begin{cases}
r+\run(B_2) & \mbox{if \(B_2\) starts with an up (resp. down) step},\\
r+\run(B_2)-1 & \mbox{if \(B_2\) starts with a down (resp. up) step}.
\end{cases}
\]
In this case, if \(m\) is odd (resp. even), then \(S=\overline{B_1}uAd\overline{C}\) (resp. \(S=B_1u\overline{A}d\overline{C}\)), where \(C=\phi_{m',k'}(\overline{B_2})\) (resp. \(C=\phi_{m',k'}(B_2)\)). By the induction hypothesis, if \(m\) is odd (resp. even), then 
\[
\run(C)=\begin{cases}
\run(B_2) & \mbox{if \(B_2\) starts with an up (resp. down) step},\\
\run(B_2)-1 & \mbox{if \(B_2\) starts with a down (resp. up) step}.\\
\end{cases}
\]
Hence, if \(m\) is odd, then
\(\run(S)=r+\run(C)-1\) so that
\[
\run(S)=\begin{cases}
r+\run(B_2)-1 & \mbox{if \(B_2\) starts with an up step},\\
r+\run(B_2)-2 & \mbox{if \(B_2\) starts with a down step}, 
\end{cases}
\]
which means that \(\run(S)=\run(P)-1\). Similarly, we show that  \(\run(S)=\run(P)\) for even $m$.

The proofs of Case 2--i)  and Case 2--ii) are similar, so we only prove Case 2--ii). We divide this case into two cases depending on the parity of $m$.

When $m$ is odd, we write $P=A_1 dA_2 d B_1 u B_2$ and $S=B_1u\overline{A_2} d \overline{C}$, where $A_1$ starts with an up step and $C=\phi_{m',k'}(\overline{A_1}uB_2)$ for some $k'<k$ and even $m'$. Note that $\run(P)=\run(A_1)+\run(dA_2 d B_1)+\run(uB_2)-2$, $\run(S)=\run(B_1u\overline{A_2} d)+\run(C)-1$, and $\run(C)=\run(\overline{A_1}uB_2)-1$ by the induction hypothesis. We have
\[
\run(P)=\begin{cases}
\run(A_1)+\run(dA_2 d B_1)+\run(B_2)-1 & \mbox{if \(B_2\) starts with a down step},\\
\run(A_1)+\run(dA_2 d B_1)+\run(B_2)-2 & \mbox{if \(B_2\) starts with an up step},
\end{cases}
\]
and
\[
\run(S)=\begin{cases}
\run(B_1 u \overline{A_2} d)+\run(A_1)+\run(B_2)-2 & \mbox{if \(B_2\) starts with a down step},\\
\run(B_1 u \overline{A_2} d)+\run(A_1)+\run(B_2)-3 & \mbox{if \(B_2\) starts with an up step}.
\end{cases}
\]
Since $\run(dA_2 d B_1)=\run(B_1 u \overline{A_2} d)$ whenever $A_2$ starts with $u$ or $d$, we get \(\run(S)=\run(P)-1\).

For even $m$, we write $P=uA_2 uB_1 d B_2$ and $S=\overline{B_1}uA_2 d \overline{C}$, where $C=\phi_{m',k'}(u\overline{B_2})$ for some $k'<k$ and even $m'$. We have $A_1=\emptyset$ because $P$ starts with $u$. We also get \(\run(S)=\run(P)\) in a similar manner.
\end{proof}

\begin{rem}\label{rem:run}
Proposition~\ref{prop:run} confirms that $|\run(P)-\run(\phi_{m,k}(P))|\le 1$, which shows that the map \(\phi_{m,k}\) is structurally distinguishable from the map from Gu and Prodinger \cite{GP}. For example, Gu and Prodinger's map sends 
\[
P=dduuuuudddddduuu \mapsto S=uuuuuuddduuuddud,
\] 
which shows that their map has paths satisfying $|\run(P)-\run(S)|=2$ (one can obtain this example by putting \(A=dddddduuu\), \(B=uu\), \(C=du\), and \(D=\emptyset\) in Figure 2.5 in \cite{GP}). 
\end{rem}

Let \(\oF(m,r,k;i)\) denote the set of paths in $\oF(m,r,k)$ with $\lceil i/2 \rceil$ downward (resp. upward) runs for odd (resp. even) \(m\), and let \(\oM(m,r,k;i)\) denote the set of paths in $\oM(m,r,k)$ with $i$ runs.
By Proposition \ref{prop:run}, it is straightforward to get the following corollary.

\begin{cor}\label{cor:run}
For non-negative integers $i$ and $m$ of the same (resp. different) parity, the map \(\phi_{m,k}\) induces a bijection between the set \(\oM(m,r,k;i)\) and the set of paths in $\oF(m,r,k;i)$ that end with a downward (resp. upward) run.
\end{cor}

\begin{rem}
It is clear that the set of paths in $\oM(m,0,k;i)$ is in bijection with the set of symmetric Dyck paths of length \(2m\) with \(i\) peaks which touch the line \(y=k\) and is contained in the strip \(0 \le y \le k\). By Corollary \ref{cor:run}, the set of symmetric Dyck paths of length \(2m\) with \(i\) peaks corresponds to the set of paths in \(\cup_{k}\oF(m,0,k;i)\) that end with a downward (resp. upward) run whenever \(i\) and \(m\) have the same (resp. different) parity. This one-to-one correspondence gives a combinatorial proof of the well-known fact that the number of symmetric Dyck paths of length \(2m\) with \(i\) peaks is given by
\begin{equation}\label{eq:symmetric_peak}
\binom{\lf \frac{m-1}{2} \rf}{\lf \frac{i-1}{2} \rf}\binom{\lf \frac{m}{2} \rf}{\lf \frac{i}{2}\rf}.
\end{equation}
\end{rem}


\section{Cornerless free Motzkin paths}

In this section, we combinatorially interpret cornerless Motzkin paths as a \(t\)-core partitions. First let us consider the restriction of the map \(\phi_{m,k}\).


\subsection{Restriction to cornerless free Motkzin paths}\label{sec:cornerless}

Recall that \(\oFc(m,r,k)\) is the set of cornerless free Motzkin paths in \(\oF(m,r,k)\) that never start with a down (resp. up) step for odd (resp. even) \(m\) and \(\oMc(m,r,k)\) is the set of cornerless Motzkin prefixes in \(\oM(m,r,k)\) that end with a flat step. Now we show that the map \(\phi_{m,k}\), defined in Section~\ref{sec:map_phi}, gives a one-to-one correspondence between these sets.

\begin{prop} \label{thm:cornerless map}
For given non-negative integers \(m,r\), and \(k\), $\phi_{m,k}$ induces a bijection between the sets \(\oFc(m,r,k)\) and \(\oMc(m,r,k)\).
\end{prop}

\begin{proof}
Let $P \in \oF(m,r,k)$ and $\phi_{m,k}(P) \in \oM(m,r,k)$.
For $k=0$ or $1$, $P \in \oFc(m,r,k)$ when it is cornerless and starts with a flat step, and $\overleftarrow{P} \in \oMc(m,r,k)$ when it is cornerless and ends with a flat step. Hence, $P \in \oFc(m,r,k)$ if and only if $\phi_{m,k}(P)=\overleftarrow{P}\in \oMc(m,r,k)$.

Let $m$ and $k$ have the same parity and there is no break step in $P$ with $k>1$.
It follows from \eqref{eq:basic} and \eqref{eq:case1-1} that $P \in \oFc(m,r,k)$ and $\phi_{m,k}(P) \in \oMc(m,r,k)$ have the same restriction such that $A$ and $B$ are cornerless, $A$ does not start with a down (resp. up) step for odd (resp. even) $m$, and $\beta>0$. Hence, $P \in \oFc(m,r,k)$ if and only if $\phi_{m,k}(P)\in \oMc(m,r,k)$.

For the remaining cases, we assume that $\phi_{m',k'}$  induces a bijection between $\oFc(m',r',k')$ and $\oMc(m',r',k')$ for $k'<k$. We consider the case when $m$ and $k$ have the same parity and there is a break step $b$ in $P$. By \eqref{eq:case1-2P} and \eqref{eq:case1-2}, $P \in \oFc(m,r,k)$ and $\phi_{m,k}(P)\in \oMc(m,r,k)$ have the same condition such that $A$ and $B_1$ are cornerless, $A$ does not start with a down (resp. up) step, $\overline{B_2}f^{\delta}$ (resp. $B_2f^{\delta}$) $\in \oFc(m',r',k')$ for some $k'<k$ when $m$ is odd (resp. even), and $\beta>0$. Hence, $P \in \oFc(m,r,k)$ if and only if $\phi_{m,k}(P)\in \oMc(m,r,k)$.

Similarly, we can show that $P \in \oFc(m,r,k)$ if and only if $\phi_{m,k}(P)\in \oMc(m,r,k)$ when $m$ and $k$ have different parity by considering \eqref{eq:case2-1P}, \eqref{eq:case2-1}, \eqref{eq:case2-2P}, and \eqref{eq:case2-2}.
\end{proof}

\subsection{Cornerless Motkzin paths and \(t\)-cores}\label{sec:tcore}

A \emph{partition} \(\la=(\la_1,\la_2,\dots,\la_{\ell})\) is a non-increasing positive integer sequence. The \emph{Young diagram} of \(\la\) is an array of boxes arranged in left-justified rows with \(\la_{i}\) boxes in the \(i\)th row. An \emph{inner corner} of a Young diagram is a box that can be removed from the Young diagram and the rest of the Young diagram is still the Young diagram of a partition. We say that \(\la\) has \(m\) corners if its Young diagram has \(m\) inner corners.
For a given Young diagram, the \emph{hook length} of a box at the position \((i,j)\), denoted by \(h(i,j)\), is the number of boxes on the right, in the below, and itself. For a partition \(\la\), the \emph{beta-set} of \(\la\), denoted by \(\beta(\la)\), is the set of hook lengths of boxes in the first column of the Young diagram of \(\la\). A partition is called a \emph{\(t\)-core} if its Young diagram has no box of hook length \(t\). 
We mainly consider \(t\)-core partitions with \(m\) corners and use the abacus diagram introduced by James and Kerber \cite{JK} to count them.
The \emph{\(t\)-abacus diagram} is a diagram to be the bottom and left-justified diagram with infinitely many rows labeled by \(i\in \mathbb{N} \cup \{0\} \) and \(t\) columns labeled by \(j=0,1,\dots,t-1\) whose position \((i,j)\) is labeled by \(ti+j\). The \emph{\(t\)-abacus} of a partition \(\la\) is obtained from the \(t\)-abacus diagram by placing a bead on each position labeled by \(h\), where \(h\in\beta(\la)\). A position without bead is called a \emph{spacer}. The following lemma is useful to determine whether a given partition is \(t\)-core or not. 

\begin{lem}\cite[Lemma 2.7.13]{JK}\label{lem:core} 
A partition \(\la\) is a \(t\)-core if and only if \(h\in\beta(\la)\) implies \(h-t\in\beta(\la)\) whenever \(h>t\).
Equivalently, \(\la\) is a \(t\)-core if and only if the \(t\)-abacus of \(\la\) has no spacer below a bead in any column. 
\end{lem}

From the above lemma, we easily obtain a simple bijection between the set of \(t\)-core partitions and the set of non-negative integer sequences \((n_0,n_1,\dots, n_{t-1})\), where \(n_0=0\) and \(n_j\) is the number of beads in column \(j\) for \(j=1,\dots,t-1\). Using the bijection between the bar graphs and cornerless Motzkin paths, introduced by Deutsch and Elizalde \cite{DE}, we give a path interpretation of the \(t\)-core partitions restricted by the number of corners and the first hook length \(h(1,1)\).

\begin{thm}\label{thm:corner} 
For non-negative integers \(t\), \(m\), and \(k\), there is a bijection between any pair of the following sets.
\begin{enumerate}
\item[(a)] The set of \(t\)-core partitions with \(m\) corners such that \(h(1,1)<kt \).
\item[(b)] The set of non-negative integer sequences \((n_0,n_1,\dots, n_{t-1})\) satisfying that \(n_0=0\), \(n_i\leq k\) for all \(i\), and 
	\[
	\sum_{i=1}^{t} |n_i-n_{i-1}| =2m,
	\]
where we set \(n_t:=0\).
\item[(c)] The set of cornerless Motzkin paths of length \(2m+t-1\) with \(t-1\) flat steps that are contained in the strip \(0\leq y \leq k\).
\end{enumerate}
\end{thm}

\begin{proof}
Let $A, B$, and $C$ be the set described in (a), (b), and (c), respectively. Set the maps $\phi_1:A\rightarrow B$ and $\phi_2:B\rightarrow C$. For a partition $\la\in A$, let $n_i$ be the number of beads in the $i$th column of the $t$-abacus of $\la$. Given $\la\in A$, define $\phi_1(\la)=(n_0, n_1, \ldots, n_{t-1})$. Then, by the definition of the $t$-abacus and the fact that $h(1,1)<kt$, it is given that $n_0=0$ and $n_i \le k$ for each $i$. Moreover, we get one inner corner for each  maximal sequence of consecutive numbers in the beta-set $\beta(\la)$. 
Note that \( \sum_{i=1}^{t} \max(n_i-n_{i-1},0) \) counts the number of hook lengths which is the smallest among each maximal sequence of consecutive numbers in the beta-set, so we get $\sum_{i=1}^{t}|n_i-n_{i-1}|=2m$. 
Let $\psi_1:B\rightarrow A$ and $\vv{n}=(n_0, n_1, \ldots, n_{t-1})\in B$. Define $\psi_1(\vv{n})=\la$, where $\la$ is the partition obtained from the $t$-abacus diagram with $n_i$ beads in the $i$th column. We place the beads on the elements of $\beta(\la)$ in the $t$-abacus diagram.
Then, since column $0$ has no bead and each $n_i\le k$ for all $i$, the largest element in $\beta(\la)$ is less than $kt$, meaning that $\la$ is a $t$-core partition with $h(1,1)<kt$. Also, the fact that the sum of $|n_i-n_{i-1}|$ is $2m$ implies that there are $m$ piles of beads which are placed on $m$ maximal consecutive numbers, so $\la$ has $m$ corners. 

For $\vv{n}=(n_0, n_1, \ldots, n_{t-1})\in B$, let $\phi_2(\vv{n})=P_{\vv{n}}$, where $P_{\vv{n}}$ be the cornerless Motzkin path which starts at $(0,0)$, ends at $(2m+t-1,0)$, and has $t-1$ flat steps at height $n_1, n_2, \cdots, n_{t-1}$ with proper up and/or down steps connecting those flat steps. Due to the fact that $n_i\le k$, it is given that $P_{\vv{n}}$ is contained in the strip $0\le y\le k$. 

Let $\psi_2:C\rightarrow B$ and $P\in C$. Define $\psi_2(P)=(n_0, n_1, \ldots, n_{t-1})$, where $n_0=0$ and, for $1\le i\le t-1$, each $n_i$ represents the height of the $i$th flat step in $P$. We know that $P$ is contained in the strip $0\le y\le k$, which implies $n_i\le k$. On the path $P$, there are $2m$ many up and down steps. The number $|n_i-n_{i-1}|$ represents the difference of the height of the $(i-1)$st flat step and the $i$th flat step, so it counts the number of up or down steps in between those two flat steps. Since \( \sum_{i=1}^{t} \max(n_i-n_{i-1},0) =  \sum_{i=1}^{t} |\min(n_i-n_{i-1},0)|\), we get $\sum_{i=1}^t|n_i-n_{i-1}|=2m$.

\end{proof}

For example, there are sixteen \(4\)-core partitions with \(2\) corners. By letting \(t=4\) and \(m=2\) in Theorem~\ref{thm:corner}, we get the correspondence between these partitions, abaci, non-negative integer sequences, and cornerless Motzkin paths as described in Figure~\ref{fig:corner_example}.

\begin{figure}[hbt]
\centering
\footnotesize{
\begin{tikzpicture}[scale=0.3]

\node at (-22,0) {};
\node[left] at (-16,0.75) {\(\ast\) (4,4,2,2)};
\node at (-15.2,0.75) {\(\leftrightarrow\)};
\draw (-13.5,0) -- (-9.5,0);

\foreach \i in {0,1,2,3}
\node[above] at (-13+\i,-0.1) {\i};

\foreach \i in {4,5,6,7}
\node[above] at (-17+\i,0.9) {\i};

\foreach \i in {2,3}
\draw[above] (-13+\i,0.5) circle (13pt);

\foreach \i in {2,3}
\draw[above] (-13+\i,1.5) circle (13pt);

\node at (-7.8,0.75) {\(\leftrightarrow\)};

\node at (-4.5,0.75) {[0,0,2,2]};

\node at (-1,0.75) {\(\leftrightarrow\)};

\draw[gray, ->] (0.5,0) -- (8,0);
\draw[gray, ->] (0.5,0) -- (0.5,2.5);

\foreach \i in {1,2,...,7}
\draw[dotted] (0.5+\i,0) -- (0.5+\i,2)
;	

\foreach \i in {1,2}
\draw[dotted] (0.5,\i) -- (7.5,0\i);

\node at (9,0) {};

\draw[thick] 
(0.5,0) -- (1.5,0) -- (2.5,1) -- (3.5,2) -- (4.5,2) -- (5.5,2) -- (6.5,1) -- (7.5,0);
\end{tikzpicture}\\

\begin{tikzpicture}[scale=0.3]

\node at (-22,0) {};
\node[left] at (-16,0.75) {(2,2,2,1,1,1)};
\node at (-15.2,0.75) {\(\leftrightarrow\)};
\draw (-13.5,0) -- (-9.5,0);

\foreach \i in {0,1,2,3}
\node[above] at (-13+\i,-0.1) {\i};

\foreach \i in {4,5,6,7}
\node[above] at (-17+\i,0.9) {\i};

\foreach \i in {1,2,3}
\draw[above] (-13+\i,0.5) circle (13pt);

\foreach \i in {1,2,3}
\draw[above] (-13+\i,1.5) circle (13pt);

\node at (-7.8,0.75) {\(\leftrightarrow\)};

\node at (-4.5,0.75) {[0,2,2,2]};

\node at (-1,0.75) {\(\leftrightarrow\)};

\draw[gray, ->] (0.5,0) -- (8,0);
\draw[gray, ->] (0.5,0) -- (0.5,2.5);

\foreach \i in {1,2,...,7}
\draw[dotted] (0.5+\i,0) -- (0.5+\i,2)
;	

\foreach \i in {1,2}
\draw[dotted] (0.5,\i) -- (7.5,0\i);

\node at (9,0) {};
\node[right] at (8,1) {\(\star\)};

\draw[thick] 
(0.5,0) -- (1.5,1) -- (2.5,2) -- (3.5,2) -- (4.5,2) -- (5.5,2) -- (6.5,1) -- (7.5,0);
\end{tikzpicture}\\

\begin{tikzpicture}[scale=0.3]

\node at (-22,0) {};
\node[left] at (-16,0.75) {(3,3,1,1,1)};
\node at (-15.2,0.75) {\(\leftrightarrow\)};
\draw (-13.5,0) -- (-9.5,0);

\foreach \i in {0,1,2,3}
\node[above] at (-13+\i,-0.1) {\i};

\foreach \i in {4,5,6,7}
\node[above] at (-17+\i,0.9) {\i};

\foreach \i in {1,2,3}
\draw[above] (-13+\i,0.5) circle (13pt);

\foreach \i in {2,3}
\draw[above] (-13+\i,1.5) circle (13pt);

\node at (-7.8,0.75) {\(\leftrightarrow\)};

\node at (-4.5,0.75) {[0,1,2,2]};

\node at (-1,0.75) {\(\leftrightarrow\)};

\draw[gray, ->] (0.5,0) -- (8,0);
\draw[gray, ->] (0.5,0) -- (0.5,2.5);

\foreach \i in {1,2,...,7}
\draw[dotted] (0.5+\i,0) -- (0.5+\i,2)
;	

\foreach \i in {1,2}
\draw[dotted] (0.5,\i) -- (7.5,0\i);

\node at (9,0) {};
\draw[thick] 
(0.5,0) -- (1.5,1) -- (2.5,1) -- (3.5,2) -- (4.5,2) -- (5.5,2) -- (6.5,1) -- (7.5,0);
\end{tikzpicture}\\

\begin{tikzpicture}[scale=0.3]

\node at (-22,0) {};
\node[left] at (-16,0.75) {(5,2,2)};
\node at (-15.2,0.75) {\(\leftrightarrow\)};
\draw (-13.5,0) -- (-9.5,0);

\foreach \i in {0,1,2,3}
\node[above] at (-13+\i,-0.1) {\i};

\foreach \i in {4,5,6,7}
\node[above] at (-17+\i,0.9) {\i};

\foreach \i in {2,3}
\draw[above] (-13+\i,0.5) circle (13pt);

\foreach \i in {3}
\draw[above] (-13+\i,1.5) circle (13pt);

\node at (-7.8,0.75) {\(\leftrightarrow\)};

\node at (-4.5,0.75) {[0,0,1,2]};

\node at (-1,0.75) {\(\leftrightarrow\)};

\draw[gray, ->] (0.5,0) -- (8,0);
\draw[gray, ->] (0.5,0) -- (0.5,2.5);

\foreach \i in {1,2,...,7}
\draw[dotted] (0.5+\i,0) -- (0.5+\i,2)
;	

\foreach \i in {1,2}
\draw[dotted] (0.5,\i) -- (7.5,0\i);

\node at (9,0) {};

\draw[thick] 
(0.5,0) -- (1.5,0) -- (2.5,1) -- (3.5,1) -- (4.5,2) -- (5.5,2) -- (6.5,1) -- (7.5,0);
\end{tikzpicture}\\

\begin{tikzpicture}[scale=0.3]

\node at (-22,0) {};
\node[left] at (-16,0.75) {(6,3)};
\node at (-15.2,0.75) {\(\leftrightarrow\)};
\draw (-13.5,0) -- (-9.5,0);

\foreach \i in {0,1,2,3}
\node[above] at (-13+\i,-0.1) {\i};

\foreach \i in {4,5,6,7}
\node[above] at (-17+\i,0.9) {\i};

\foreach \i in {3}
\draw[above] (-13+\i,0.5) circle (13pt);

\foreach \i in {3}
\draw[above] (-13+\i,1.5) circle (13pt);

\node at (-7.8,0.75) {\(\leftrightarrow\)};

\node at (-4.5,0.75) {[0,0,0,2]};

\node at (-1,0.75) {\(\leftrightarrow\)};

\draw[gray, ->] (0.5,0) -- (8,0);
\draw[gray, ->] (0.5,0) -- (0.5,2.5);

\foreach \i in {1,2,...,7}
\draw[dotted] (0.5+\i,0) -- (0.5+\i,2)
;	

\foreach \i in {1,2}
\draw[dotted] (0.5,\i) -- (7.5,0\i);

\node at (9,0) {};

\draw[thick] 
(0.5,0) -- (1.5,0) -- (2.5,0) -- (3.5,1) -- (4.5,2) -- (5.5,2) -- (6.5,1) -- (7.5,0);
\end{tikzpicture}\\

\begin{tikzpicture}[scale=0.3]

\node at (-22,0) {};
\node[left] at (-16,0.75) {(3,3,1,1)};
\node at (-15.2,0.75) {\(\leftrightarrow\)};
\draw (-13.5,0) -- (-9.5,0);

\foreach \i in {0,1,2,3}
\node[above] at (-13+\i,-0.1) {\i};

\foreach \i in {4,5,6,7}
\node[above] at (-17+\i,0.9) {\i};

\foreach \i in {1,2}
\draw[above] (-13+\i,0.5) circle (13pt);

\foreach \i in {1,2}
\draw[above] (-13+\i,1.5) circle (13pt);

\node at (-7.8,0.75) {\(\leftrightarrow\)};

\node at (-4.5,0.75) {[0,2,2,0]};

\node at (-1,0.75) {\(\leftrightarrow\)};

\draw[gray, ->] (0.5,0) -- (8,0);
\draw[gray, ->] (0.5,0) -- (0.5,2.5);

\foreach \i in {1,2,...,7}
\draw[dotted] (0.5+\i,0) -- (0.5+\i,2)
;	

\foreach \i in {1,2}
\draw[dotted] (0.5,\i) -- (7.5,0\i);

\node at (9,0) {};

\draw[thick] 
(0.5,0) -- (1.5,1) -- (2.5,2) -- (3.5,2) -- (4.5,2) -- (5.5,1) -- (6.5,0) -- (7.5,0);
\end{tikzpicture}\\

\begin{tikzpicture}[scale=0.3]

\node at (-22,0) {};
\node[left] at (-16,0.75) {(4,2,2)};
\node at (-15.2,0.75) {\(\leftrightarrow\)};
\draw (-13.5,0) -- (-9.5,0);

\foreach \i in {0,1,2,3}
\node[above] at (-13+\i,-0.1) {\i};

\foreach \i in {4,5,6,7}
\node[above] at (-17+\i,0.9) {\i};

\foreach \i in {2,3}
\draw[above] (-13+\i,0.5) circle (13pt);

\foreach \i in {2}
\draw[above] (-13+\i,1.5) circle (13pt);

\node at (-7.8,0.75) {\(\leftrightarrow\)};

\node at (-4.5,0.75) {[0,0,2,1]};

\node at (-1,0.75) {\(\leftrightarrow\)};

\draw[gray, ->] (0.5,0) -- (8,0);
\draw[gray, ->] (0.5,0) -- (0.5,2.5);

\foreach \i in {1,2,...,7}
\draw[dotted] (0.5+\i,0) -- (0.5+\i,2)
;	

\foreach \i in {1,2}
\draw[dotted] (0.5,\i) -- (7.5,0\i);

\node at (9,0) {};

\draw[thick] 
(0.5,0) -- (1.5,0) -- (2.5,1) -- (3.5,2) -- (4.5,2) -- (5.5,1) -- (6.5,1) -- (7.5,0);
\end{tikzpicture}\\

\begin{tikzpicture}[scale=0.3]

\node at (-22,0) {};
\node[left] at (-16,0.75) {(2,2,1,1,1)};
\node at (-15.2,0.75) {\(\leftrightarrow\)};
\draw (-13.5,0) -- (-9.5,0);

\foreach \i in {0,1,2,3}
\node[above] at (-13+\i,-0.1) {\i};

\foreach \i in {4,5,6,7}
\node[above] at (-17+\i,0.9) {\i};

\foreach \i in {1,2,3}
\draw[above] (-13+\i,0.5) circle (13pt);

\foreach \i in {1,2}
\draw[above] (-13+\i,1.5) circle (13pt);

\node at (-7.8,0.75) {\(\leftrightarrow\)};

\node at (-4.5,0.75) {[0,2,2,1]};

\node at (-1,0.75) {\(\leftrightarrow\)};

\draw[gray, ->] (0.5,0) -- (8,0);
\draw[gray, ->] (0.5,0) -- (0.5,2.5);

\foreach \i in {1,2,...,7}
\draw[dotted] (0.5+\i,0) -- (0.5+\i,2)
;	

\foreach \i in {1,2}
\draw[dotted] (0.5,\i) -- (7.5,0\i);

\node at (9,0) {};

\draw[thick] 
(0.5,0) -- (1.5,1) -- (2.5,2) -- (3.5,2) -- (4.5,2) -- (5.5,1) -- (6.5,1) -- (7.5,0);
\end{tikzpicture}\\

\begin{tikzpicture}[scale=0.3]

\node at (-22,0) {};
\node[left] at (-16,0.75) {\(\ast\) (4,1,1,1)};
\node at (-15.2,0.75) {\(\leftrightarrow\)};
\draw (-13.5,0) -- (-9.5,0);

\foreach \i in {0,1,2,3}
\node[above] at (-13+\i,-0.1) {\i};

\foreach \i in {4,5,6,7}
\node[above] at (-17+\i,0.9) {\i};

\foreach \i in {1,2,3}
\draw[above] (-13+\i,0.5) circle (13pt);

\foreach \i in {3}
\draw[above] (-13+\i,1.5) circle (13pt);

\node at (-7.8,0.75) {\(\leftrightarrow\)};

\node at (-4.5,0.75) {[0,1,1,2]};

\node at (-1,0.75) {\(\leftrightarrow\)};

\draw[gray, ->] (0.5,0) -- (8,0);
\draw[gray, ->] (0.5,0) -- (0.5,2.5);

\foreach \i in {1,2,...,7}
\draw[dotted] (0.5+\i,0) -- (0.5+\i,2)
;	

\foreach \i in {1,2}
\draw[dotted] (0.5,\i) -- (7.5,0\i);

\node at (9,0) {};

\draw[thick] 
(0.5,0) -- (1.5,1) -- (2.5,1) -- (3.5,1) -- (4.5,2) -- (5.5,2) -- (6.5,1) -- (7.5,0);
\end{tikzpicture}\\

\begin{tikzpicture}[scale=0.3]

\node at (-22,0) {};
\node[left] at (-16,0.75) {(5,2)};
\node at (-15.2,0.75) {\(\leftrightarrow\)};
\draw (-13.5,0) -- (-9.5,0);

\foreach \i in {0,1,2,3}
\node[above] at (-13+\i,-0.1) {\i};

\foreach \i in {4,5,6,7}
\node[above] at (-17+\i,0.9) {\i};

\foreach \i in {2}
\draw[above] (-13+\i,0.5) circle (13pt);

\foreach \i in {2}
\draw[above] (-13+\i,1.5) circle (13pt);

\node at (-7.8,0.75) {\(\leftrightarrow\)};

\node at (-4.5,0.75) {[0,0,2,0]};

\node at (-1,0.75) {\(\leftrightarrow\)};

\draw[gray, ->] (0.5,0) -- (8,0);
\draw[gray, ->] (0.5,0) -- (0.5,2.5);

\foreach \i in {1,2,...,7}
\draw[dotted] (0.5+\i,0) -- (0.5+\i,2)
;	

\foreach \i in {1,2}
\draw[dotted] (0.5,\i) -- (7.5,0\i);

\node at (9,0) {};
\node[right] at (8,1) {\(\star\)};

\draw[thick] 
(0.5,0) -- (1.5,0) -- (2.5,1) -- (3.5,2) -- (4.5,2) -- (5.5,1) -- (6.5,0) -- (7.5,0);
\end{tikzpicture}\\

\begin{tikzpicture}[scale=0.3]

\node at (-22,0) {};
\node[left] at (-16,0.75) {(3,1,1,1)};
\node at (-15.2,0.75) {\(\leftrightarrow\)};
\draw (-13.5,0) -- (-9.5,0);

\foreach \i in {0,1,2,3}
\node[above] at (-13+\i,-0.1) {\i};

\foreach \i in {4,5,6,7}
\node[above] at (-17+\i,0.9) {\i};

\foreach \i in {1,2,3}
\draw[above] (-13+\i,0.5) circle (13pt);

\foreach \i in {2}
\draw[above] (-13+\i,1.5) circle (13pt);

\node at (-7.8,0.75) {\(\leftrightarrow\)};

\node at (-4.5,0.75) {[0,1,2,1]};

\node at (-1,0.75) {\(\leftrightarrow\)};

\draw[gray, ->] (0.5,0) -- (8,0);
\draw[gray, ->] (0.5,0) -- (0.5,2.5);

\foreach \i in {1,2,...,7}
\draw[dotted] (0.5+\i,0) -- (0.5+\i,2)
;	

\foreach \i in {1,2}
\draw[dotted] (0.5,\i) -- (7.5,0\i);

\node at (9,0) {};
\node[right] at (8,1) {\(\star\)};

\draw[thick] 
(0.5,0) -- (1.5,1) -- (2.5,1) -- (3.5,2) -- (4.5,2) -- (5.5,1) -- (6.5,1) -- (7.5,0);
\end{tikzpicture}\\

\begin{tikzpicture}[scale=0.3]

\node at (-22,0) {};
\node[left] at (-16,0.75) {(4,1,1)};
\node at (-15.2,0.75) {\(\leftrightarrow\)};
\draw (-13.5,0) -- (-9.5,0);

\foreach \i in {0,1,2,3}
\node[above] at (-13+\i,-0.1) {\i};

\foreach \i in {4,5,6,7}
\node[above] at (-17+\i,0.9) {\i};

\foreach \i in {1,2}
\draw[above] (-13+\i,0.5) circle (13pt);

\foreach \i in {2}
\draw[above] (-13+\i,1.5) circle (13pt);

\node at (-7.8,0.75) {\(\leftrightarrow\)};

\node at (-4.5,0.75) {[0,1,2,0]};

\node at (-1,0.75) {\(\leftrightarrow\)};

\draw[gray, ->] (0.5,0) -- (8,0);
\draw[gray, ->] (0.5,0) -- (0.5,2.5);

\foreach \i in {1,2,...,7}
\draw[dotted] (0.5+\i,0) -- (0.5+\i,2)
;	

\foreach \i in {1,2}
\draw[dotted] (0.5,\i) -- (7.5,0\i);

\node at (9,0) {};

\draw[thick] 
(0.5,0) -- (1.5,1) -- (2.5,1) -- (3.5,2) -- (4.5,2) -- (5.5,1) -- (6.5,0) -- (7.5,0);
\end{tikzpicture}\\

\begin{tikzpicture}[scale=0.3]

\node at (-22,0) {};
\node[left] at (-16,0.75) {(2,1,1,1)};
\node at (-15.2,0.75) {\(\leftrightarrow\)};
\draw (-13.5,0) -- (-9.5,0);

\foreach \i in {0,1,2,3}
\node[above] at (-13+\i,-0.1) {\i};

\foreach \i in {4,5,6,7}
\node[above] at (-17+\i,0.9) {\i};

\foreach \i in {1,2,3}
\draw[above] (-13+\i,0.5) circle (13pt);

\foreach \i in {1}
\draw[above] (-13+\i,1.5) circle (13pt);

\node at (-7.8,0.75) {\(\leftrightarrow\)};

\node at (-4.5,0.75) {[0,2,1,1]};

\node at (-1,0.75) {\(\leftrightarrow\)};

\draw[gray, ->] (0.5,0) -- (8,0);
\draw[gray, ->] (0.5,0) -- (0.5,2.5);

\foreach \i in {1,2,...,7}
\draw[dotted] (0.5+\i,0) -- (0.5+\i,2)
;	

\foreach \i in {1,2}
\draw[dotted] (0.5,\i) -- (7.5,0\i);

\node at (9,0) {};

\draw[thick] 
(0.5,0) -- (1.5,1) -- (2.5,2) -- (3.5,2) -- (4.5,1) -- (5.5,1) -- (6.5,1) -- (7.5,0);
\end{tikzpicture}\\

\begin{tikzpicture}[scale=0.3]

\node at (-22,0) {};
\node[left] at (-16,0.75) {\(\ast\) (3,1,1)};
\node at (-15.2,0.75) {\(\leftrightarrow\)};
\draw (-13.5,0) -- (-9.5,0);

\foreach \i in {0,1,2,3}
\node[above] at (-13+\i,-0.1) {\i};

\foreach \i in {4,5,6,7}
\node[above] at (-17+\i,0.9) {\i};

\foreach \i in {1,2}
\draw[above] (-13+\i,0.5) circle (13pt);

\foreach \i in {1}
\draw[above] (-13+\i,1.5) circle (13pt);

\node at (-7.8,0.75) {\(\leftrightarrow\)};

\node at (-4.5,0.75) {[0,2,1,0]};

\node at (-1,0.75) {\(\leftrightarrow\)};

\draw[gray, ->] (0.5,0) -- (8,0);
\draw[gray, ->] (0.5,0) -- (0.5,2.5);

\foreach \i in {1,2,...,7}
\draw[dotted] (0.5+\i,0) -- (0.5+\i,2)
;	

\foreach \i in {1,2}
\draw[dotted] (0.5,\i) -- (7.5,0\i);

\node at (9,0) {};

\draw[thick] 
(0.5,0) -- (1.5,0) -- (2.5,1) -- (3.5,2) -- (4.5,2) -- (5.5,2) -- (6.5,1) -- (7.5,0);
\end{tikzpicture}\\

\begin{tikzpicture}[scale=0.3]

\node at (-22,0) {};
\node[left] at (-16,0.75) {(4,1)};
\node at (-15.2,0.75) {\(\leftrightarrow\)};
\draw (-13.5,0) -- (-9.5,0);

\foreach \i in {0,1,2,3}
\node[above] at (-13+\i,-0.1) {\i};

\foreach \i in {4,5,6,7}
\node[above] at (-17+\i,0.9) {\i};

\foreach \i in {1}
\draw[above] (-13+\i,0.5) circle (13pt);

\foreach \i in {1}
\draw[above] (-13+\i,1.5) circle (13pt);

\node at (-7.8,0.75) {\(\leftrightarrow\)};

\node at (-4.5,0.75) {[0,2,0,0]};

\node at (-1,0.75) {\(\leftrightarrow\)};

\draw[gray, ->] (0.5,0) -- (8,0);
\draw[gray, ->] (0.5,0) -- (0.5,2.5);

\foreach \i in {1,2,...,7}
\draw[dotted] (0.5+\i,0) -- (0.5+\i,2)
;	

\foreach \i in {1,2}
\draw[dotted] (0.5,\i) -- (7.5,0\i);

\node at (9,0) {};

\draw[thick] 
(0.5,0) -- (1.5,1) -- (2.5,2) -- (3.5,2) -- (4.5,1) -- (5.5,0) -- (6.5,0) -- (7.5,0);
\end{tikzpicture}\\

\begin{tikzpicture}[scale=0.3]

\node at (-22,0) {};
\node[left] at (-16,0.75) {\(\ast\) (2,1)};
\node at (-15.2,0.75) {\(\leftrightarrow\)};
\draw (-13.5,0) -- (-9.5,0);

\foreach \i in {0,1,2,3}
\node[above] at (-13+\i,-0.1) {\i};

\foreach \i in {4,5,6,7}
\node[above] at (-17+\i,0.9) {\i};

\foreach \i in {1,3}
\draw[above] (-13+\i,0.5) circle (13pt);

\node at (-7.8,0.75) {\(\leftrightarrow\)};

\node at (-4.5,0.75) {[0,1,0,1]};

\node at (-1,0.75) {\(\leftrightarrow\)};

\draw[gray, ->] (0.5,0) -- (8,0);
\draw[gray, ->] (0.5,0) -- (0.5,2.5);

\foreach \i in {1,2,...,7}
\draw[dotted] (0.5+\i,0) -- (0.5+\i,2)
;	

\foreach \i in {1,2}
\draw[dotted] (0.5,\i) -- (7.5,0\i);

\node at (9,0) {};
\node[right] at (8,1) {\(\star\)};

\draw[thick] 
(0.5,0) -- (1.5,1) -- (2.5,1) -- (3.5,0) -- (4.5,0) -- (5.5,1) -- (6.5,1) -- (7.5,0);
\end{tikzpicture}\\
}
\caption{\(4\)-cores with \(2\) corners and the corresponding objects}
\label{fig:corner_example}
\end{figure}
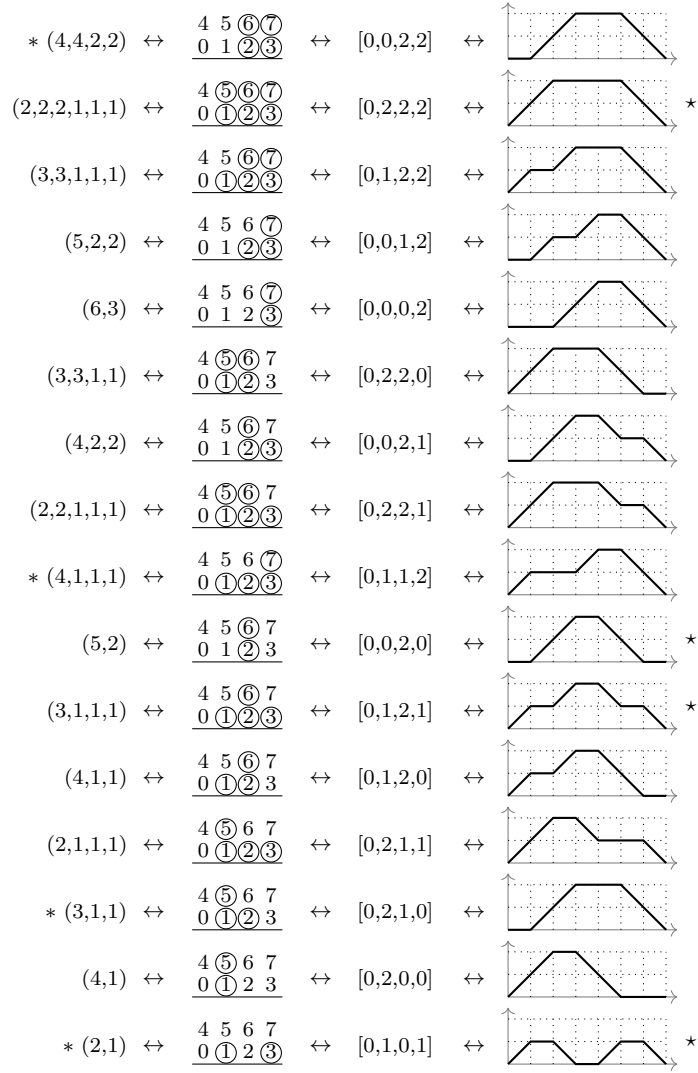

We denote that a partition $\la$ is a $(t_1, t_2, \ldots, t_p)$-core if $\la$ is a $t_i$-core for all $i=1,\dots,p$. It is known that the number of $t$-core partitions is infinite, and the number of $(t_1, t_2, \ldots, t_p)$-cores is finite for relatively prime $t_1,\dots,t_p$. Huang and Wang \cite{HW} enumerated the number of $(t,t+1)$-cores, $(t,t+1,t+2)$-cores with the fixed number of corners, where these results are generalized to $(t,t+1, \cdots, t+p)$-cores in \cite{CHS}.
As far as we know, it seems new to get the formula for the number of $t$-core partitions with the fixed number of corners, which we enumerate this by using the path interpretation. 

\begin{prop}\label{prop:count_ordinary}
The number of \(t\)-core partitions with \(m\) corners is given by
\[
\cc(t,m):=\sum_{i=1}^{\min(m,\lf t/2 \rf)} N(m,i) \binom{t+2m-2i}{2m},
\]
where \(N(m,i)=\frac{1}{m}\binom{m}{i}\binom{m}{i-1}\) denotes the Narayana number.
\end{prop}

\begin{proof}
By Theorem~\ref{thm:corner}, \(\cc(t,m)\) is equal to the number of cornerless Motzkin paths of length \(2m+t-1\) with \(t-1\) flat steps. Let a Dyck path consisting of \(m\) up steps and \(m\) down steps with \(i\) peaks be given. The number of ways of inserting \(t-1\) flat steps such that the resultant path becomes a cornerless Motzkin path is \(\binom{t+2m-2i}{2m}\) since we have to insert at least one flat steps at the positions of \(i\) peaks and \(i-1\) valleys. As the number of Dyck paths consisting of \(m\) up steps and \(m\) down steps with \(i\) peaks is counted by the Narayana number \(N(m,i)\), the proof is followed.
\end{proof}

The numbers of \(t\)-core partitions with \(m\) corners for \(2\leq t\leq 6\) and \(1\leq m \leq 8\) are given in  Table~\ref{table}. Clearly, \(\cc(2,m)=1\), \(\cc(3,m)=2m+1\), and \(\cc(4,m)=(5m^2+5m+2)/2\).
See sequences A063490 and A160747 in \cite{OEIS} for more the values of \(\cc(t,m)\) for \(t=5\) and \(t=6\), respectively. 

\begin{table}[htb!]
\centering
\begin{tabular}{c|cccccccccccccccc}
\noalign{\smallskip}\noalign{\smallskip}
\(t \backslash m\)&& 1 && 2 && 3 && 4 && 5 && 6 && 7 && 8 \\
\hline
2 && 1 && 1 && 1 && 1 && 1 && 1 && 1 && 1\\
3 && 3 && 5 && 7 && 9 && 11 && 13 && 15 && 17\\
4 && 6 && 16 && 31 && 51 && 76 && 106 && 141 && 181\\
5 && 10 && 40 && 105 && 219 && 396 && 650 && 995 && 1445\\
6 && 15 && 85 && 295 && 771 && 1681 && 3235 && 5685 && 9325\\

\end{tabular}
\caption{The numbers \(\cc(t,m)\) of \(t\)-cores with \(m\) corners}\label{table}
\end{table}


\subsection{Cornerless symmetric Motzkin paths and self-conjugate \(t\)-cores}\label{sec:self}

For a partition \(\la\), its \emph{conjugate} is the partition \(\la'=(\la'_1,\la'_2,\dots)\), where each \(\la'_j\) is the number of boxes in the \(j\)th column of the Young diagram of \(\la\). A partition \(\la\) is called \emph{self-conjugate} if \(\la=\la'\).
Let \(\MD(\la)\) denote the set of the main diagonal hook lengths of \(\la\). Note that if \(\la\) is a self-conjugate partition, then the elements in \(\MD(\la)\) are all distinct and odd. Similar to Lemma~\ref{lem:core}, Ford, Mai, and Sze \cite{FMS} gave a useful result to determine whether a given partition is self-conjugate \(t\)-core or not.

\begin{prop}\cite[Proposition 3]{FMS}\label{prop:self-conjugate}
Let \(\la\) be a self-conjugate partition. Then \(\la\) is a \(t\)-core if and only if both of the following hold:
\begin{enumerate}
    \item[(a)] For \(h>t\), if \(h\in\MD(\la)\), then \(h-2t\in\MD(\la)\).
    \item[(b)] If \(h_1,h_2\in\MD(\la)\), then \(h_1+h_2\neq 2t\).
\end{enumerate}
\end{prop}

We slightly modify the \(t\)-abacus to get the \emph{\(t\)-doubled abacus}, which is useful when we deal with a self-conjugate \(t\)-core partition. Let the \emph{\(t\)-doubled abacus diagram} is a left-justified diagram with infinitely many rows labeled by \(i\in\mathbb{Z}\) and \(\lf t/2 \rf\) columns labeled by \(j=0,1,\dots,\lf t/2\rf -1\) whose position \((i,j)\) is labeled by \(|2(ti+j)+1|\). The \emph{\(t\)-doubled abacus} of a self-conjugate partition \(\la\) is obtained from the \(t\)-doubled abacus diagram by placing a bead on each position labeled by \(h\), where \(h\in\MD(\la)\). From Proposition~\ref{prop:self-conjugate}, we have the following lemma. 

\begin{lem}\label{lem:self-conjugate}
A self-conjugate partition \(\la\) is \(t\)-core if and only if the \(t\)-doubled abacus diagram of \(\la\) satisfies both of the following.
\begin{enumerate}
    \item[(a)] If a bead is placed on position \((i,j)\) with \(i>0\) (resp. \(i<0\)), then a bead is also placed on position \((0,j)\) (resp. \((-1,j)\)) and there is no spacer between them in any column \(j\).
    \item[(b)] A bead can be placed on at most one of the two positions \((-1,j)\) and \((0,j)\) in any column \(j\).
\end{enumerate}
\end{lem}

From the above lemma, we easily obtain a simple bijection between the set of self-conjugate \(t\)-core partitions and the set of integer sequences \((n_{0},\dots,n_{\lf t/2\rf-1})\), where the number of beads in column \(j\) is denoted by either \(n_j\) or \(-n_j\) for \(j=0,1,\dots,\lf t/2 \rf -1\) if a bead is placed in position \((0,j)\) or not, respectively. Now we give a path interpretation of the self-conjugate \(t\)-core partitions restricted by the number of corners and the first hook length \(h(1,1)\). We define 
\[
\F_c(m,r,k):=\bigcup_{i=0}^{k}\oF_c(m,r,i) \quad \text{and} \quad \M_c(m,r,k):=\bigcup_{i=0}^{k}\oM_c(m,r,i).
\]

\begin{thm}\label{thm:self-conjugate} 
For non-negative integers \(t\), \(m\), and \(k\), there is a bijection between any pair of the following sets.
\begin{enumerate}
\item[(a)] The set of self-conjugate \(t\)-cores with \(m\) corners such that \(h(1,1)<kt\).
\item[(b)] The set of integer sequences \((n_0,n_1,\dots, n_{\lf t/2 \rf-1})\) satisfying that for odd (resp. even) \(m\), \(n_0\) is positive (resp. non-positive); for all \(i\), \(-\lf k/2 \rf \leq n_i\leq \lf (k+1)/2\rf\); and
	\[
	\sum_{i=0}^{\lf t/2 \rf} |n_i-n_{i-1}| =\begin{cases}
	m+1 & \text{for odd~} m,\\
	m & \text{for even~}m,
	\end{cases}
	\]
where we set \(n_{-1}:=0\) and \(n_{\lf t/2 \rf}:=0\).
\item[(c)] The set of cornerless free Motzkin paths in \(\Fc(m,\lf t/2 \rf ,k)\).
\item[(d)] The set of cornerless Motzkin prefixes in \(\MC(m,\lf t/2 \rf ,k)\).
\item[(e)] The set of cornerless symmetric Motzkin paths of length \(2m+t-1\) with \(t-1\) flat steps that are contained in the strip \(0\leq y \leq k\).
\end{enumerate}
\end{thm}

\begin{proof}
Let $A, B, C, D,$ and $E$ be the set described in (a), (b), (c), (d), and (e), respectively. By similar argument to the proof of Proposition~\ref{thm:cornerless map}, we know that there is a bijection between $C$ and $D$.
Now we set $\phi_1:A\rightarrow B, \phi_2:B\rightarrow C$, and $\phi_3:D\rightarrow E$ and show that $\phi_1, \phi_2, \phi_3$ are bijections.

Given $\la\in A$, let $\phi_1(A)=(n_0, n_1, \ldots, n_{\lf t/2 \rf-1})$, where each $n_i$ is the highest or lowest row that the bead is placed in the $i$th column depending on the sign of $n_i$. We get that $1\in MD(\la)$ when the number of corners $m$ is odd and $1\not\in MD(\la)$ otherwise. Thus, $n_0$ is positive when $m$ is odd and negative otherwise. This map gives a bijection between $A$ and $B$.

Let $\vv{n}=(n_0, n_1, \ldots, n_{\lf t/2 \rf-1})$. For odd (resp. even) $m$, let $\phi_2(\vv{n})$ be the cornerless free Motzkin path that starts at $(0,1)$ (resp. $(0,0)$), ends at $(m+\lf t/2\rf,0)$, has $i$th flat step at height $n_{i-1}$ with proper up and down steps between them. Then, the map $\phi_2$ describes a bijection between $B$ and $C$. 

Denote a path by $P=p_1p_2\cdots p_{m+\lf t/2 \rf}\in D$. We set
\[
\phi_3(P)=\begin{cases}
p_1p_2\cdots p_{m+\lf t/2 \rf}p_{m+\lf t/2 \rf}\cdots p_2p_1 &\text{if \(m\) is odd,}\\
p_1p_2\cdots p_{m+\lf t/2 \rf-1}p_{m+\lf t/2 \rf}p_{m+\lf t/2 \rf-1}\cdots p_2p_1 &\text{if \(m\) is even.}
\end{cases}
\]
Then, the map $\phi_3$ is a bijective. 
\end{proof}

Note that Figure~\ref{fig:corner_example} shows that there are four self-conjugate \(4\)-core partitions with \(2\) corners and four cornerless symmetric Motzkin paths of length \(7\) with \(3\) flat steps, which are marked by \(\ast\) and \(\star\), respectively. The correspondences between the sets described in Theorem~\ref{thm:self-conjugate} for \(t=4, m=2\) and \(t=5, m=3\) are given in Figure~\ref{fig:self-conjugate_example}.

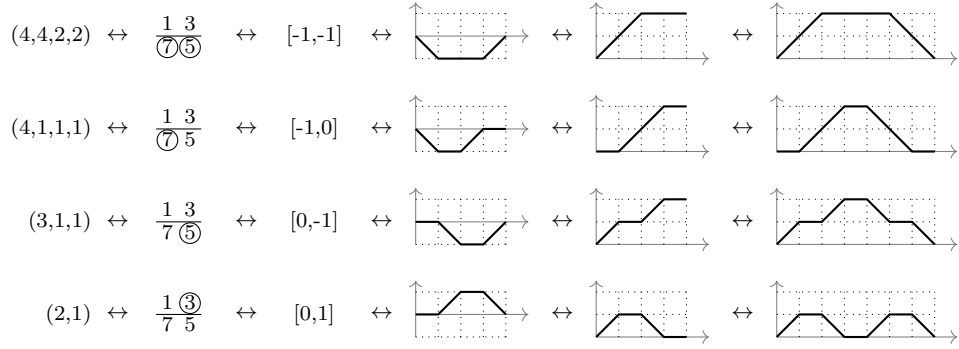
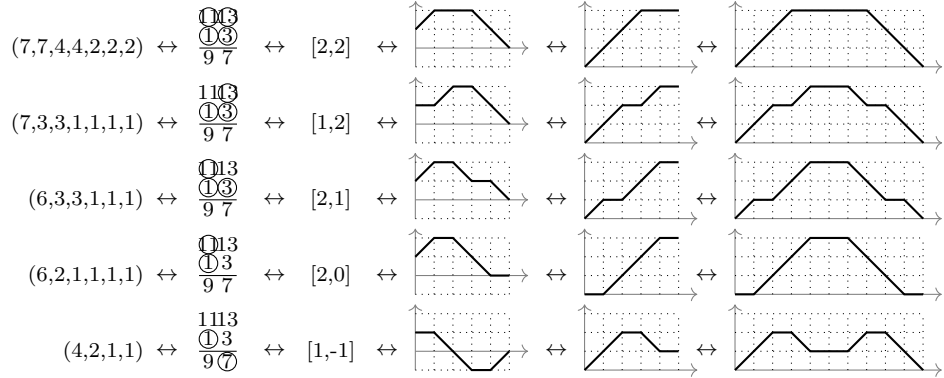
\begin{figure}[hbt]
\centering

\begin{subfigure}[b]{\textwidth}
\footnotesize{
\begin{tikzpicture}[scale=0.3]

\node at (-20,0) {};
\node[left] at (-16,0) {(4,4,2,2)};
\node at (-15.2,0) {\(\leftrightarrow\)};
\draw (-13.5,0) -- (-11.5,0);

\foreach \i/\j in {0/1,1/3}
\node[above] at (-13+\i,-0.1) {\j};

\foreach \i/\j in {0/7,1/5}
\node[above] at (-13+\i,-1.1) {\j};

\foreach \i in {}
\draw[above] (-13+\i,0.5) circle (13pt);

\foreach \i in {0,1}
\draw[above] (-13+\i,-0.5) circle (13pt);

\node at (-9.5,0) {\(\leftrightarrow\)};

\node at (-6.5,0) {[-1,-1]};

\node at (-3.5,0) {\(\leftrightarrow\)};

\draw[gray, ->] (-2,0) -- (3,0);
\draw[gray, ->] (-2,-1) -- (-2,1.5);

\foreach \i in {1,2,...,4}
\draw[dotted] (-2+\i,-1) -- (-2+\i,1);	

\foreach \i in {-1,1}
\draw[dotted] (-2,\i) -- (2,\i);

\node at (9,0) {};

\draw[thick] 
(-2,0) -- (-1,-1) -- (0,-1) -- (1,-1) -- (2,0);

\node at (4.5,0) {\(\leftrightarrow\)};

\draw[gray, ->] (6,-1) -- (11,-1);
\draw[gray, ->] (6,-1) -- (6,1.5);

\foreach \i in {1,2,...,4}
\draw[dotted] (6+\i,-1) -- (6+\i,1);	

\foreach \i in {0,1}
\draw[dotted] (6,\i) -- (10,\i);

\draw[thick] 
(6,-1) -- (7,0) -- (8,1) -- (9,1) -- (10,1);

\node at (12.5,0) {\(\leftrightarrow\)};

\draw[gray, ->] (14,-1) -- (22,-1);
\draw[gray, ->] (14,-1) -- (14,1.5);

\foreach \i in {1,2,...,7}
\draw[dotted] (14+\i,-1) -- (14+\i,1);	

\foreach \i in {0,1}
\draw[dotted] (14,\i) -- (21,\i);

\draw[thick] 
(14,-1) -- (15,0) -- (16,1) -- (17,1) -- (18,1) -- (19,1) -- (20,0) -- (21,-1);

\end{tikzpicture}\\

\begin{tikzpicture}[scale=0.3]

\node at (-20,0) {};
\node[left] at (-16,0) {(4,1,1,1)};
\node at (-15.2,0) {\(\leftrightarrow\)};
\draw (-13.5,0) -- (-11.5,0);

\foreach \i/\j in {0/1,1/3}
\node[above] at (-13+\i,-0.1) {\j};

\foreach \i/\j in {0/7,1/5}
\node[above] at (-13+\i,-1.1) {\j};

\foreach \i in {}
\draw[above] (-13+\i,0.5) circle (13pt);

\foreach \i in {0}
\draw[above] (-13+\i,-0.5) circle (13pt);

\node at (-9.5,0) {\(\leftrightarrow\)};

\node at (-6.5,0) {[-1,0]};

\node at (-3.5,0) {\(\leftrightarrow\)};

\draw[gray, ->] (-2,0) -- (3,0);
\draw[gray, ->] (-2,-1) -- (-2,1.5);

\foreach \i in {1,2,...,4}
\draw[dotted] (-2+\i,-1) -- (-2+\i,1);	

\foreach \i in {-1,1}
\draw[dotted] (-2,\i) -- (2,\i);

\node at (9,0) {};

\draw[thick] 
(-2,0) -- (-1,-1) -- (0,-1) -- (1,0) -- (2,0);

\node at (4.5,0) {\(\leftrightarrow\)};

\draw[gray, ->] (6,-1) -- (11,-1);
\draw[gray, ->] (6,-1) -- (6,1.5);

\foreach \i in {1,2,...,4}
\draw[dotted] (6+\i,-1) -- (6+\i,1);	

\foreach \i in {0,1}
\draw[dotted] (6,\i) -- (10,\i);

\draw[thick] 
(6,-1) -- (7,-1) -- (8,0) -- (9,1) -- (10,1);

\node at (12.5,0) {\(\leftrightarrow\)};

\draw[gray, ->] (14,-1) -- (22,-1);
\draw[gray, ->] (14,-1) -- (14,1.5);

\foreach \i in {1,2,...,7}
\draw[dotted] (14+\i,-1) -- (14+\i,1);	

\foreach \i in {0,1}
\draw[dotted] (14,\i) -- (21,\i);

\draw[thick] 
(14,-1) -- (15,-1) -- (16,0) -- (17,1) -- (18,1) -- (19,0) -- (20,-1) -- (21,-1);

\end{tikzpicture}\\

\begin{tikzpicture}[scale=0.3]

\node at (-20,0) {};
\node[left] at (-16,0) {(3,1,1)};
\node at (-15.2,0) {\(\leftrightarrow\)};
\draw (-13.5,0) -- (-11.5,0);

\foreach \i/\j in {0/1,1/3}
\node[above] at (-13+\i,-0.1) {\j};

\foreach \i/\j in {0/7,1/5}
\node[above] at (-13+\i,-1.1) {\j};

\foreach \i in {}
\draw[above] (-13+\i,0.5) circle (13pt);

\foreach \i in {1}
\draw[above] (-13+\i,-0.5) circle (13pt);

\node at (-9.5,0) {\(\leftrightarrow\)};

\node at (-6.5,0) {[0,-1]};

\node at (-3.5,0) {\(\leftrightarrow\)};

\draw[gray, ->] (-2,0) -- (3,0);
\draw[gray, ->] (-2,-1) -- (-2,1.5);

\foreach \i in {1,2,...,4}
\draw[dotted] (-2+\i,-1) -- (-2+\i,1);	

\foreach \i in {-1,1}
\draw[dotted] (-2,\i) -- (2,\i);

\node at (9,0) {};

\draw[thick] 
(-2,0) -- (-1,0) -- (0,-1) -- (1,-1) -- (2,0);

\node at (4.5,0) {\(\leftrightarrow\)};

\draw[gray, ->] (6,-1) -- (11,-1);
\draw[gray, ->] (6,-1) -- (6,1.5);

\foreach \i in {1,2,...,4}
\draw[dotted] (6+\i,-1) -- (6+\i,1);	

\foreach \i in {0,1}
\draw[dotted] (6,\i) -- (10,\i);

\draw[thick] 
(6,-1) -- (7,0) -- (8,0) -- (9,1) -- (10,1);

\node at (12.5,0) {\(\leftrightarrow\)};

\draw[gray, ->] (14,-1) -- (22,-1);
\draw[gray, ->] (14,-1) -- (14,1.5);

\foreach \i in {1,2,...,7}
\draw[dotted] (14+\i,-1) -- (14+\i,1);	

\foreach \i in {0,1}
\draw[dotted] (14,\i) -- (21,\i);

\draw[thick] 
(14,-1) -- (15,0) -- (16,0) -- (17,1) -- (18,1) -- (19,0) -- (20,0) -- (21,-1);

\end{tikzpicture}\\

\begin{tikzpicture}[scale=0.3]

\node at (-20,0) {};
\node[left] at (-16,0) {(2,1)};
\node at (-15.2,0) {\(\leftrightarrow\)};
\draw (-13.5,0) -- (-11.5,0);

\foreach \i/\j in {0/1,1/3}
\node[above] at (-13+\i,-0.1) {\j};

\foreach \i/\j in {0/7,1/5}
\node[above] at (-13+\i,-1.1) {\j};

\foreach \i in {1}
\draw[above] (-13+\i,0.5) circle (13pt);

\foreach \i in {}
\draw[above] (-13+\i,-0.5) circle (13pt);

\node at (-9.5,0) {\(\leftrightarrow\)};

\node at (-6.5,0) {[0,1]};

\node at (-3.5,0) {\(\leftrightarrow\)};

\draw[gray, ->] (-2,0) -- (3,0);
\draw[gray, ->] (-2,-1) -- (-2,1.5);

\foreach \i in {1,2,...,4}
\draw[dotted] (-2+\i,-1) -- (-2+\i,1);	

\foreach \i in {-1,1}
\draw[dotted] (-2,\i) -- (2,\i);

\node at (9,0) {};

\draw[thick] 
(-2,0) -- (-1,0) -- (0,1) -- (1,1) -- (2,0);

\node at (4.5,0) {\(\leftrightarrow\)};

\draw[gray, ->] (6,-1) -- (11,-1);
\draw[gray, ->] (6,-1) -- (6,1.5);

\foreach \i in {1,2,...,4}
\draw[dotted] (6+\i,-1) -- (6+\i,1);	

\foreach \i in {0,1}
\draw[dotted] (6,\i) -- (10,\i);

\draw[thick] 
(6,-1) -- (7,0) -- (8,0) -- (9,-1) -- (10,-1);

\node at (12.5,0) {\(\leftrightarrow\)};

\draw[gray, ->] (14,-1) -- (22,-1);
\draw[gray, ->] (14,-1) -- (14,1.5);

\foreach \i in {1,2,...,7}
\draw[dotted] (14+\i,-1) -- (14+\i,1);	

\foreach \i in {0,1}
\draw[dotted] (14,\i) -- (21,\i);

\draw[thick] 
(14,-1) -- (15,0) -- (16,0) -- (17,-1) -- (18,-1) -- (19,0) -- (20,0) -- (21,-1);

\end{tikzpicture}\\

}

\subcaption{\(t=4\) and \(m=2\)}
\end{subfigure}
\vspace{2mm}

\begin{subfigure}[b]{\textwidth}
\centering
\footnotesize{
\begin{tikzpicture}[scale=0.25]

\node at (-23.5,0) {};
\node[left] at (-16,0) {(7,7,4,4,2,2,2)};
\node at (-15.2,0) {\(\leftrightarrow\)};
\draw (-13.5,0) -- (-11.5,0);

\foreach \i/\j in {0/11,1/13}
\node[above] at (-13+\i,0.9) {\j};

\foreach \i/\j in {0/1,1/3}
\node[above] at (-13+\i,-0.1) {\j};

\foreach \i/\j in {0/9,1/7}
\node[above] at (-13+\i,-1.3) {\j};

\foreach \i in {0,1}
\draw[above] (-13+\i,1.65) circle (14pt);

\foreach \i in {0,1}
\draw[above] (-13+\i,0.65) circle (14pt);

\foreach \i in {}
\draw[above] (-13+\i,-0.55) circle (14pt);

\node at (-9.5,0) {\(\leftrightarrow\)};

\node at (-6.5,0) {[2,2]};

\node at (-3.5,0) {\(\leftrightarrow\)};

\draw[gray, ->] (-2,0) -- (4,0);
\draw[gray, ->] (-2,-1) -- (-2,2.5);

\foreach \i in {1,2,...,5}
\draw[dotted] (-2+\i,-1) -- (-2+\i,2);	

\foreach \i in {-1,1,2}
\draw[dotted] (-2,\i) -- (3,\i);

\node at (9,0) {};

\draw[thick] 
(-2,1) -- (-1,2) -- (0,2) -- (1,2) -- (2,1) -- (3,0);

\node at (5.5,0) {\(\leftrightarrow\)};

\draw[gray, ->] (7,-1) -- (13,-1);
\draw[gray, ->] (7,-1) -- (7,2.5);

\foreach \i in {1,2,...,5}
\draw[dotted] (7+\i,-1) -- (7+\i,2);	

\foreach \i in {0,1,2}
\draw[dotted] (7,\i) -- (12,\i);

\draw[thick] 
(7,-1) -- (8,0) -- (9,1) -- (10,2) -- (11,2) -- (12,2);

\node at (13.5,0) {\(\leftrightarrow\)};

\draw[gray, ->] (15,-1) -- (26,-1);
\draw[gray, ->] (15,-1) -- (15,2.5);

\foreach \i in {1,2,...,10}
\draw[dotted] (15+\i,-1) -- (15+\i,2);	

\foreach \i in {0,1,2}
\draw[dotted] (15,\i) -- (25,\i);

\draw[thick] 
(15,-1) -- (16,0) -- (17,1) -- (18,2) -- (19,2) -- (20,2) -- (21,2) -- (22,2) -- (23,1) -- (24,0) -- (25,-1);

\end{tikzpicture}\\

\begin{tikzpicture}[scale=0.25]

\node at (-23.5,0) {};
\node[left] at (-16,0) {(7,3,3,1,1,1,1)};
\node at (-15.2,0) {\(\leftrightarrow\)};
\draw (-13.5,0) -- (-11.5,0);

\foreach \i/\j in {0/11,1/13}
\node[above] at (-13+\i,0.9) {\j};

\foreach \i/\j in {0/1,1/3}
\node[above] at (-13+\i,-0.1) {\j};

\foreach \i/\j in {0/9,1/7}
\node[above] at (-13+\i,-1.3) {\j};

\foreach \i in {1}
\draw[above] (-13+\i,1.65) circle (14pt);

\foreach \i in {0,1}
\draw[above] (-13+\i,0.65) circle (14pt);

\foreach \i in {}
\draw[above] (-13+\i,-0.55) circle (14pt);

\node at (-9.5,0) {\(\leftrightarrow\)};

\node at (-6.5,0) {[1,2]};

\node at (-3.5,0) {\(\leftrightarrow\)};

\draw[gray, ->] (-2,0) -- (4,0);
\draw[gray, ->] (-2,-1) -- (-2,2.5);

\foreach \i in {1,2,...,5}
\draw[dotted] (-2+\i,-1) -- (-2+\i,2);	

\foreach \i in {-1,1,2}
\draw[dotted] (-2,\i) -- (3,\i);

\node at (9,0) {};

\draw[thick] 
(-2,1) -- (-1,1) -- (0,2) -- (1,2) -- (2,1) -- (3,0);

\node at (5.5,0) {\(\leftrightarrow\)};

\draw[gray, ->] (7,-1) -- (13,-1);
\draw[gray, ->] (7,-1) -- (7,2.5);

\foreach \i in {1,2,...,5}
\draw[dotted] (7+\i,-1) -- (7+\i,2);	

\foreach \i in {0,1,2}
\draw[dotted] (7,\i) -- (12,\i);

\draw[thick] 
(7,-1) -- (8,0) -- (9,1) -- (10,1) -- (11,2) -- (12,2);

\node at (13.5,0) {\(\leftrightarrow\)};

\draw[gray, ->] (15,-1) -- (26,-1);
\draw[gray, ->] (15,-1) -- (15,2.5);

\foreach \i in {1,2,...,10}
\draw[dotted] (15+\i,-1) -- (15+\i,2);	

\foreach \i in {0,1,2}
\draw[dotted] (15,\i) -- (25,\i);

\draw[thick] 
(15,-1) -- (16,0) -- (17,1) -- (18,1) -- (19,2) -- (20,2) -- (21,2) -- (22,1) -- (23,1) -- (24,0) -- (25,-1);

\end{tikzpicture}\\

\begin{tikzpicture}[scale=0.25]

\node at (-23.5,0) {};
\node[left] at (-16,0) {(6,3,3,1,1,1)};
\node at (-15.2,0) {\(\leftrightarrow\)};
\draw (-13.5,0) -- (-11.5,0);

\foreach \i/\j in {0/11,1/13}
\node[above] at (-13+\i,0.9) {\j};

\foreach \i/\j in {0/1,1/3}
\node[above] at (-13+\i,-0.1) {\j};

\foreach \i/\j in {0/9,1/7}
\node[above] at (-13+\i,-1.3) {\j};

\foreach \i in {0}
\draw[above] (-13+\i,1.65) circle (14pt);

\foreach \i in {0,1}
\draw[above] (-13+\i,0.65) circle (14pt);

\foreach \i in {}
\draw[above] (-13+\i,-0.55) circle (14pt);

\node at (-9.5,0) {\(\leftrightarrow\)};

\node at (-6.5,0) {[2,1]};

\node at (-3.5,0) {\(\leftrightarrow\)};

\draw[gray, ->] (-2,0) -- (4,0);
\draw[gray, ->] (-2,-1) -- (-2,2.5);

\foreach \i in {1,2,...,5}
\draw[dotted] (-2+\i,-1) -- (-2+\i,2);	

\foreach \i in {-1,1,2}
\draw[dotted] (-2,\i) -- (3,\i);

\node at (9,0) {};

\draw[thick] 
(-2,1) -- (-1,2) -- (0,2) -- (1,1) -- (2,1) -- (3,0);

\node at (5.5,0) {\(\leftrightarrow\)};

\draw[gray, ->] (7,-1) -- (13,-1);
\draw[gray, ->] (7,-1) -- (7,2.5);

\foreach \i in {1,2,...,5}
\draw[dotted] (7+\i,-1) -- (7+\i,2);	

\foreach \i in {0,1,2}
\draw[dotted] (7,\i) -- (12,\i);

\draw[thick] 
(7,-1) -- (8,0) -- (9,0) -- (10,1) -- (11,2) -- (12,2);

\node at (13.5,0) {\(\leftrightarrow\)};

\draw[gray, ->] (15,-1) -- (26,-1);
\draw[gray, ->] (15,-1) -- (15,2.5);

\foreach \i in {1,2,...,10}
\draw[dotted] (15+\i,-1) -- (15+\i,2);	

\foreach \i in {0,1,2}
\draw[dotted] (15,\i) -- (25,\i);

\draw[thick] 
(15,-1) -- (16,0) -- (17,0) -- (18,1) -- (19,2) -- (20,2) -- (21,2) -- (22,1) -- (23,0) -- (24,0) -- (25,-1);

\end{tikzpicture}\\

\begin{tikzpicture}[scale=0.25]

\node at (-23.5,0) {};
\node[left] at (-16,0) {(6,2,1,1,1,1)};
\node at (-15.2,0) {\(\leftrightarrow\)};
\draw (-13.5,0) -- (-11.5,0);

\foreach \i/\j in {0/11,1/13}
\node[above] at (-13+\i,0.9) {\j};

\foreach \i/\j in {0/1,1/3}
\node[above] at (-13+\i,-0.1) {\j};

\foreach \i/\j in {0/9,1/7}
\node[above] at (-13+\i,-1.3) {\j};

\foreach \i in {0}
\draw[above] (-13+\i,1.65) circle (14pt);

\foreach \i in {0}
\draw[above] (-13+\i,0.65) circle (14pt);

\foreach \i in {}
\draw[above] (-13+\i,-0.55) circle (14pt);

\node at (-9.5,0) {\(\leftrightarrow\)};

\node at (-6.5,0) {[2,0]};

\node at (-3.5,0) {\(\leftrightarrow\)};

\draw[gray, ->] (-2,0) -- (4,0);
\draw[gray, ->] (-2,-1) -- (-2,2.5);

\foreach \i in {1,2,...,5}
\draw[dotted] (-2+\i,-1) -- (-2+\i,2);	

\foreach \i in {-1,1,2}
\draw[dotted] (-2,\i) -- (3,\i);

\node at (9,0) {};

\draw[thick] 
(-2,1) -- (-1,2) -- (0,2) -- (1,1) -- (2,0) -- (3,0);

\node at (5.5,0) {\(\leftrightarrow\)};

\draw[gray, ->] (7,-1) -- (13,-1);
\draw[gray, ->] (7,-1) -- (7,2.5);

\foreach \i in {1,2,...,5}
\draw[dotted] (7+\i,-1) -- (7+\i,2);	

\foreach \i in {0,1,2}
\draw[dotted] (7,\i) -- (12,\i);

\draw[thick] 
(7,-1) -- (8,-1) -- (9,0) -- (10,1) -- (11,2) -- (12,2);

\node at (13.5,0) {\(\leftrightarrow\)};

\draw[gray, ->] (15,-1) -- (26,-1);
\draw[gray, ->] (15,-1) -- (15,2.5);

\foreach \i in {1,2,...,10}
\draw[dotted] (15+\i,-1) -- (15+\i,2);	

\foreach \i in {0,1,2}
\draw[dotted] (15,\i) -- (25,\i);

\draw[thick] 
(15,-1) -- (16,-1) -- (17,0) -- (18,1) -- (19,2) -- (20,2) -- (21,2) -- (22,1) -- (23,0) -- (24,-1) -- (25,-1);

\end{tikzpicture}\\

\begin{tikzpicture}[scale=0.25]

\node at (-23.5,0) {};
\node[left] at (-16,0) {(4,2,1,1)};
\node at (-15.2,0) {\(\leftrightarrow\)};
\draw (-13.5,0) -- (-11.5,0);

\foreach \i/\j in {0/11,1/13}
\node[above] at (-13+\i,0.9) {\j};

\foreach \i/\j in {0/1,1/3}
\node[above] at (-13+\i,-0.1) {\j};

\foreach \i/\j in {0/9,1/7}
\node[above] at (-13+\i,-1.3) {\j};

\foreach \i in {}
\draw[above] (-13+\i,1.65) circle (14pt);

\foreach \i in {0}
\draw[above] (-13+\i,0.65) circle (14pt);

\foreach \i in {1}
\draw[above] (-13+\i,-0.55) circle (14pt);

\node at (-9.5,0) {\(\leftrightarrow\)};

\node at (-6.5,0) {[1,-1]};

\node at (-3.5,0) {\(\leftrightarrow\)};

\draw[gray, ->] (-2,0) -- (4,0);
\draw[gray, ->] (-2,-1) -- (-2,2.5);

\foreach \i in {1,2,...,5}
\draw[dotted] (-2+\i,-1) -- (-2+\i,2);	

\foreach \i in {-1,1,2}
\draw[dotted] (-2,\i) -- (3,\i);

\node at (9,0) {};

\draw[thick] 
(-2,1) -- (-1,1) -- (0,0) -- (1,-1) -- (2,-1) -- (3,0);

\node at (5.5,0) {\(\leftrightarrow\)};

\draw[gray, ->] (7,-1) -- (13,-1);
\draw[gray, ->] (7,-1) -- (7,2.5);

\foreach \i in {1,2,...,5}
\draw[dotted] (7+\i,-1) -- (7+\i,2);	

\foreach \i in {0,1,2}
\draw[dotted] (7,\i) -- (12,\i);

\draw[thick] 
(7,-1) -- (8,0) -- (9,1) -- (10,1) -- (11,0) -- (12,0);

\node at (13.5,0) {\(\leftrightarrow\)};

\draw[gray, ->] (15,-1) -- (26,-1);
\draw[gray, ->] (15,-1) -- (15,2.5);

\foreach \i in {1,2,...,10}
\draw[dotted] (15+\i,-1) -- (15+\i,2);	

\foreach \i in {0,1,2}
\draw[dotted] (15,\i) -- (25,\i);

\draw[thick] 
(15,-1) -- (16,0) -- (17,1) -- (18,1) -- (19,0) -- (20,0) -- (21,0) -- (22,1) -- (23,1) -- (24,0) -- (25,-1);

\end{tikzpicture}\\
}
\subcaption{\(t=5\) and \(m=3\)}
\end{subfigure}

\caption{Examples of self-conjugate \(t\)-cores with \(m\) corners and the corresponding objects}
\label{fig:self-conjugate_example}
\end{figure}

Although the number of self-conjugate $(t,t+1, \cdots, t+p)$-cores with the fixed number of corners is unknown in general, it is enumerated in \cite{ChoHong, ChoHuh} when $p=1, 2,$ and $3$. 
The number of self-conjugate \(t\)-core partitions with \(m\) corners can be counted by using these path interpretations.

\begin{prop}\label{prop:count_self-conjugate}
The number of self-conjugate \(t\)-core partitions with \(m\) corners is given by
\[
\scc(t,m):=\sum_{i=1}^{\min(m,\lf t/2 \rf)}\binom{\lf \frac{m-1}{2} \rf}{\lf \frac{i-1}{2} \rf}\binom{\lf \frac{m}{2} \rf}{\lf \frac{i}{2}\rf}\binom{\lf \frac{t}{2} \rf +m-i}{m}
\]
for \(m>0\) and \(\scc(t,0)=1\).
In addition, \(\scc(t,m)=\scc(t+1,m)\) for even \(t\).
\end{prop}

\begin{proof}
By Theorem~\ref{thm:self-conjugate}, \(\scc(t,m)\) also counts the number of cornerless symmetric Motzkin paths of length \(2m+t-1\) with \(t-1\) flat steps. Let a symmetric Dyck path consisting of \(m\) up steps and \(m\) down steps with \(i\) peaks with \(2i\leq t\) be given. The number of ways inserting \(t-1\) flat steps such that the resultant path becomes a cornerless symmetric Motzkin path is \(\binom{\lf t/2 \rf +m-i}{m}\). The proof is followed since the number of symmetric Dyck paths consisting of \(m\) up steps and \(m\) down steps with \(i\) peaks is given by \eqref{eq:symmetric_peak}.
\end{proof}

The numbers of self-conjugate \(t\)-core partitions with \(m\) corners for \(2\leq t\leq 11\) and \(1\leq m \leq 8\) are given in  Table~\ref{table2}. 
Clearly, \(\scc(2,m)=\scc(3,m)=1\), \(\scc(4,m)=\scc(5,m)=\lf 3m/2 \rf +1\), and \(\scc(6,m)=\scc(7,m)=(10m(m+1)+(-1)^m(2m+1)+7)/8\).

\begin{table}[htb!]
\centering
\begin{tabular}{c|cccccccccccccccc}
\noalign{\smallskip}\noalign{\smallskip}
\(t \backslash m\)&& 1 && 2 && 3 && 4 && 5 && 6 && 7 && 8 \\
\hline
2,3 && 1 && 1 && 1 && 1 && 1 && 1 && 1 && 1\\
4,5 && 2 && 4 && 5 && 7 && 8 && 10 && 11 && 13\\
6,7 && 3 && 9 && 15 && 27 && 37 && 55 && 69 && 93\\
8,9 && 4 && 16 && 34 && 76 && 124 && 216 && 309 && 471\\
10,11 && 5 && 25 && 65 && 175 && 335 && 675 && 1095 && 1875\\

\end{tabular}
\caption{The numbers \(\scc(t,m)\) of self-conjugate \(t\)-cores with \(m\) corners}\label{table2}
\end{table}



\end{document}